\documentclass[12pt,leqno]{amsart}
\usepackage{a4,amssymb,amsthm,amscd,amsmath,verbatim,url,enumerate,color,cleveref}

\renewcommand{\ker}{\mathrm{Ker}}
\newcommand{\la}{\langle}
\newcommand{\ra}{\rangle}

\newcommand{\nat}{\mathbb{N}}
\newcommand{\zz}{\mathbb Z}

\newcommand{\mg}[1]{#1^{\times}}
\newcommand{\sq}[1]{#1^{\times 2}}
\newcommand{\scg}[1]{\mg{#1}/\sq{#1}}

\newcommand{\follows}[2]{$(#1\Rightarrow #2)$}
\DeclareMathOperator{\rad}{\mathsf{rad}}
\newcommand{\Nrd}{\mathsf{Nrd}}

\newcommand{\Trd}{\mathsf{Trd}}

\newcommand{\wt}[1]{\widetilde{{#1}}}

\newcommand{\lra}{\rightarrow}

\newcommand{\Ad}{{\mathsf{Ad}}}
\newcommand{\ad}{{\mathsf{ad}}}

\newcommand{\op}{{\mathsf{op}}}
\newcommand{\id}{\mathsf{id}}
\newcommand{\s}{\sigma}
\newcommand{\matr}[1]{\mathbb{M}_{#1}}

\newcommand{\can}{\mathsf{can}}

\newcommand{\Pf}{\mathfrak{P}}
\DeclareMathOperator{\car}{\mathsf{char}}
\DeclareMathOperator{\Symm}{\mathsf{Sym}}
\DeclareMathOperator{\Skew}{\mathsf{Skew}}
\DeclareMathOperator{\Symd}{\mathsf{Symd}}
\DeclareMathOperator{\Alt}{\mathsf{Alt}}

\DeclareMathOperator{\kap}{\mathsf{cap}}

\DeclareMathOperator{\Tr}{\mathsf{T}}

\DeclareMathOperator{\ind}{\mathsf{ind}}
\DeclareMathOperator{\coind}{\mathsf{coind}}
\DeclareMathOperator{\Int}{\mathsf{Int}}
\DeclareMathOperator{\End}{\mathsf{End}}
\DeclareMathOperator{\sw}{\mathsf{sw}}

\renewcommand{\dim}{\mathsf{dim}}
\renewcommand{\deg}{\mathsf{deg}}
\renewcommand{\exp}{\mathsf{exp}}
\DeclareMathOperator{\Z}{\mathsf{Z}}
\DeclareMathOperator{\C}{\mathsf{C}}
\DeclareMathOperator{\N}{\mathsf{N}}
\renewcommand{\ker}{\mathsf{Ker}}

\renewcommand{\leq}{\leqslant}

\newcommand{\vf}{\varphi}

\renewcommand{\log}{\mathsf{log}}

\swapnumbers
\newtheorem{thm}{Theorem}
\numberwithin{thm}{section}
\numberwithin{equation}{thm}
\newtheorem*{thm*}{Theorem}
\newtheorem*{mt}{Main Theorem}
\newtheorem{prop}[thm]{Proposition}
\newtheorem{cor}[thm]{Corollary}

\newtheorem*{conj*}{Conjecture}
\newtheorem{lem}[thm]{Lemma}

\theoremstyle{definition}

\newtheorem{ex}[thm]{Example}
\newtheorem{exs}[thm]{Examples}
\newtheorem{rem}[thm]{Remark}
\newtheorem{rems}[thm]{Remarks}
\renewenvironment{proof}{\par\noindent {\em Proof: }}{\hfill$\Box$\medskip}
\theoremstyle{plain}

\title[A discriminant Pfister form]{The discriminant Pfister form of
  an algebra with involution of capacity four} 
\date{7 September, 2024}
\author[K.J. Becher]{Karim Johannes Becher}
\author[N. Grenier-Boley]{Nicolas Grenier-Boley} 
\author[J.-P. Tignol]{Jean-Pierre Tignol}
\address{Universiteit Antwerpen, Departement Wiskunde,
  Middelheimlaan~1, B-2020 Antwerpen, Belgium} 
\email{karimjohannes.becher@uantwerpen.be}
\address{Normandie Univ, UNIROUEN, LDAR (EA 4434), 76000 Rouen,
  France,\linebreak Universit\'es de Paris, Artois, Cergy Pontoise,
  Paris-Est Cr\'eteil.} 
\email{nicolas.grenier-boley@univ-rouen.fr}
\address{UCLouvain, ICTEAM Institute, Avenue
  G.~Lema\^{\i}tre 4, Box~L4.05.01,
B-1348 Louvain-la-Neuve, Belgium.}
\email{jean-pierre.tignol@uclouvain.be}

\thanks{This work was supported by the FWO Odysseus Programme (project
  \emph{Explicit Methods in Quadratic Form Theory}), funded by the
  Fonds Wetenschappelijk Onderzoek -- Vlaanderen. The third author
  acknowledges support from the Fonds de la Recherche
  Scientifique--FNRS under CDR grants 1.5054.12F, J.0014.15,
  J.0149.17, J.0159.19. Work on this paper was initiated in~2010 while
the first and the third author were, respectively, Fellow and Senior
Fellow of the Zukunftskolleg, Universit\"at Konstanz, whose
hospitality is gratefully acknowledged.}

\begin{document}

\begin{abstract}\noindent
  To an orthogonal or unitary involution on a central simple algebra
  of degree $4$, or to a symplectic involution on a central simple
  algebra of degree $8$, we associate a Pfister form that
  characterises the decomposability of the algebra with involution.
  In this way we obtain a unified approach to known decomposability
  criteria for several cases, and a new result for symplectic
  involutions on degree-$8$ algebras in characteristic $2$.

  \smallskip
  \noindent
  \emph{Classification (MSC 2010):} 11E04, 11E81, 12G05, 16H05, 16R50,
  16W10

\end{abstract}

\maketitle

\section{Introduction} 

The decomposability of central simple algebras with involution into
tensor products of quaternion algebras with involution has been in the
focus of much study, motivated notably by the analogy between
decomposable involutions and quadratic Pfister forms; see~\cite{BPQ},
\cite{B}, and the references in~\cite{THyd}. When the characteristic
is different from~$2$, vanishing of the first cohomological invariant
yields a necessary condition for decomposability, which is also
sufficient for algebras of low degree; see~\cite{THyd}.  The aim of
this article is to establish in arbitrary characteristic a
decomposability criterion for algebras with involution of low degree
in terms of a canonically associated quadratic Pfister form.
The condition on the degree depends on the type of the
  involution, and can be better expressed in terms of the capacity of
  the Jordan algebra of symmetric elements, which is the maximal
  dimension of an \'etale algebra consisting of symmetric elements
  (see Section~\ref{sec:cap} and~\cite[p.~389]{BGBT18a}). In
  a nutshell, our main result functorially associates to every algebra
  with involution of capacity four a quadratic Pfister form that
  detects when the algebra decomposes into a tensor product of
  quaternion algebras with involution.

The following statement spells out explicitly our main result in the various cases
 that we consider:

\begin{mt}
  Let $F$ be a field.  Let $n\in\{1,2,3\}$.  Let $A$ be an $F$-algebra
  with $\dim_FA=2^{n+3}$ and $\s$ an $F$-linear involution on $A$ such
  that the following holds, depending on $n$:
  \begin{enumerate}
  \item[$n=1:$] $\car F\neq 2$,
   $(A,\s)$ is a central simple $F$-algebra with orthogonal involution of
    degree $4$.
\item[$n=2:$]
  $(A,\s)$ is a central simple $F$-algebra with unitary
    involution of degree~$4$.
\item[$n=3:$]
  $(A,\s)$ is a central simple $F$-algebra with symplectic
    involution of degree~$8$.
  \end{enumerate}
  Then  to $(A,\s)$ a quadratic
  $n$-fold Pfister form $\Pf_\s$ over $F$  is associated which has the following
  characteristic property:\\[1mm] 
  For any field extension $F'/F$,
  the form $\Pf_\s$ is hyperbolic over
  $F'$ if and only if the $F'$-algebra with involution
  $(A,\s)_{F'}$ obtained from $(A,\s)$ by scalar extension to $F'$
  is totally decomposable.
 \end{mt}

Here, we call $(A,\s)$  \emph{totally decomposable} if
$
    (A,\s)\simeq (Q_1\otimes \cdots\otimes Q_r,\s_1\otimes\cdots\otimes \s_r)
$
where $r=\lceil\frac{n}{2}\rceil$ and where, for $1\leq i\leq r$, $Q_i$ is a quaternion algebra over the centre of $A$ and $\s_i$ an $F$-linear involution on $Q_i$.

The notion of a central simple $F$-algebra with unitary involution (in case $n=2$) is to be understood in the sense of~\cite[\S2.B]{BOI}: the centre is a quadratic \'etale $F$-algebra.

In most of the cases the Main Theorem gives a reinterpretation of some
previously known decomposability criteria in terms of quadratic
Pfister forms.  Here our principal aim is to handle these different
cases by a new uniform approach.  In the case where $n=3$ and $\car F=2$
the result is itself novel.

Note that when $n=1$ we assume that $\car F\neq 2$.  
In this case a criterion for decomposability was established by
Knus--Parimala--Sridharan \cite{KPS}.  In~\cite{KPS2} the same authors
proved an analogous statement in arbitrary characteristic.  
This criterion could also be formulated in terms of a bilinear $1$-fold
Pfister form (given by the determinant of the involution,
see~\cite[\S7.A]{BOI}). On the other hand, a corresponding criterion
for degree-$4$ algebras with quadratic pair can be obtained using the
exceptional isomorphism $\mathsf{D}_2=\mathsf{A}_1\times\mathsf{A}_1$;
see~\cite[(15.12)]{BOI}.  This criterion could be formulated in terms
of a quadratic $1$-fold Pfister form.

The case $n=1$ is the least difficult and interesting one, but at the
same time it would be the most cumbersome to cover if characteristic
$2$ were included, in view of the necessary distinction between
orthogonal involutions and quadratic pairs.  For this reason we
decided to include the case $n=1$ only when $\car F\neq 2$, mainly in
order to highlight the analogy with the other two cases.
Nevertheless, we will have cause to consider orthogonal
involutions in characteristic~$2$ in some auxiliary results.
 
In the case $n=2$, a criterion for decomposability was first obtained
by Karpenko--Qu\'eguiner~\cite{KQ}.  Their result is stated in terms
of the discriminant algebra, and it is obtained using the exceptional
isomorphism $\mathsf{A}_3=\mathsf{D}_3$.  The quadratic $2$-fold Pfister form
$\Pf_\sigma$ turns out to be the norm form of the $F$-quaternion
algebra that is Brauer equivalent to the discriminant algebra of
$(A,\sigma)$, see \Cref{P:cap4unitary-invarel}.  
For the case $n=3$ a criterion for decomposability was
established in characteristic different from~$2$ by
Garibaldi--Parimala--Tignol~\cite{GPT} in terms of a cohomology class
of degree $3$, which gives the first nontrivial cohomological
invariant of $(A,\sigma)$.  This class is given by the
Arason invariant of the $3$-fold Pfister form $\Pf_\sigma$,
see \Cref{P:GPT}. 

What unites the three cases in the Main Theorem, in spite of the different
dimensions of the algebra, is the fact that biquadratic \'etale
subalgebras on which $\s$ restricts to the identity are maximal for
this property, by~\cite[Theorem~4.1]{BGBT18a}.

The core of the proof of the Main Theorem is carried out in
\Cref{sec:biquad}.  The construction of the quadratic form
$\Pf_\sigma$ is inspired by the treatment in \cite{GPT} of the
symplectic case in characteristic different from two.  It further
relies on a peculiar property of certain biquadratic \'etale
subalgebras, which was first used by Rost--Serre--Tignol~\cite{RST} to
define a cohomological invariant of degree~$4$ for central simple
algebras of degree~$4$.

A priori the construction depends on the choice of a biquadratic
\'etale $F$-subalgebra $L$ of $A$ contained in the space
$\Symd(\s)=\{x+\s(x)\mid x\in A\}$. In
 \Cref{L:mults} we
show that such an $L$ 
induces a decomposition
 $$\Symd(\s)=L\oplus W_1\oplus W_2\oplus W_3$$ 
where each of the $F$-subspaces $W_1$, $W_2$, $W_3$ is naturally
endowed with a quadratic form (determined up to a similarity factor).
Moreover, we show that these three quadratic forms are related by a
composition formula.  Hence they are similar to a quadratic Pfister
form, which is determined by $\s$ and $L$ and which we denote by
$\Pf_{\sigma,L}$.  In \Cref{P:Pf-dec-equiv} we show that this form
$\Pf_{\sigma,L}$ is hyperbolic if and only if the algebra with
involution decomposes along~$L$ (this notion is introduced in \Cref{S:dec}).  
We then prove in \Cref{C:Pf-disc-def} that this Pfister form does not
depend on the choice of the biquadratic subalgebra $L$.  
Hence it is an invariant of $\s$, which we denote by $\Pf_\s$.
The fact that $\Pf_\s$ has the properties stipulated in the Main
Theorem is then established by \Cref{T:Pf-disc-hyp}. 
 
Showing the independence of the Pfister form $\Pf_\s$ from the choices
made in its construction 
is the most delicate part of the proof of the Main Theorem.
This is based on a reduction to the case where $\s$ is hyperbolic, and
hence it relies on a comprehensive study of the decomposability of
algebras with hyperbolic involution. 
This is carried out in some more generality in
\Cref{S:hyperquatfactor}, and then specialised in
\Cref{sec:cap4} to the situation of capacity~$4$.

Our treatment also leads to a new result on the possible
decompositions of a totally decomposable algebra with involution
$(A,\s)$ such as in the Main Theorem.  In
\Cref{C:dec-along-all-quad-neat} we obtain that any biquadratic
\'etale $F$-subalgebra of $A$ to which $\s$ restricts to the identity
can be distributed nicely over the quaternion factors of a certain
decomposition of $(A,\s)$.  In \Cref{C:kap4ttdec-splitfactor} we
further show that, when $A$ is simple but contains zero-divisors (or
when $A$ is unitary of inner type and the simple components contain
zero-divisors), then a total decomposition of $(A,\s)$ can be found
which contains a split $F$-quaternion algebra.

In the final \Cref{S:final} we relate the quadratic Pfister forms
arising in the three cases of the Main Theorem between each other and
to various previously known invariants. We further give some examples
where the Pfister form invariant of an algebra with involution is
computed explicitly.

\bigskip
{\bf Acknowledgements.} We are grateful to the referee for a large number of
  suggestions that allowed us to improve the presentation significantly. We would further like to thank Kader Bing\"ol for various comments.

\section{Algebras with involution} 

In this section we recall some basic facts and objects associated with
involutions on central simple algebras.  We recall the distinction of
involutions into two kinds and further into three different types. We
also include some notation from \cite{BGBT18a}.  Our main reference
for involutions is \cite{BOI}.
\medbreak

Let $F$ always be a field.  Let $A$ be an $F$-algebra.  We denote by
$\Z(A)$ the centre of $A$.  For an $F$-subalgebra $B$ of $A$ we denote
by $\C_A(B)$ the centraliser of $B$ in~$A$.  Assume now that
$\dim_F(A)<\infty$.  If $A$ is simple, then $\Z(A)$ is a field and $A$
is a central simple $\Z(A)$-algebra.  In this case we denote
respectively by $\deg A$, $\ind A$ and $\exp A$ the degree, the index
and the exponent of $A$, and we further set
$\coind A=\frac{\deg A}{\ind A}$, and we call $A$ \emph{split} if
$\ind A=1$.  A central simple $F$-algebra of degree $2$ is called an
\emph{$F$-quaternion algebra}.

By an \emph{$F$-involution on $A$} we mean an $F$-linear anti-automorphism $\sigma:A\to A$ such that $\sigma\circ\sigma=\id_A$.
Consider an $F$-involution $\s$ on $A$. We set
\begin{eqnarray*}
\Symm(\sigma) & = & \{x\in A\mid \sigma(x)=\phantom{-}x\}\,,\\
\Skew(\sigma) & = & \{x\in A\mid \sigma(x)=-x\}\,,\\
\Symd(\sigma) & = & \{x+\sigma(x)\mid x\in A\}\,.
\end{eqnarray*}
Using the $F$-linear map
$A\to A,x\mapsto x+\s(x)$ one obtains that $$\dim_FA=\dim_F\Skew(\s)+\dim_F\Symd(\s).$$

An \emph{$F$-algebra with involution} is a pair $(A,\s)$ of a
finite-dimensional $F$-algebra $A$ and an $F$-involution $\s$ on $A$
such that $F= \Z(A)\cap \Symm(\s)$ and such that $A$ has no nontrivial
two-sided ideals $I$ with $\s(I)=I$. 

For an \'etale extension $L/F$ we denote by $[L:F]$ the dimension of $L$ as an $F$-vector space.

In the sequel let $(A,\s)$ be an $F$-algebra with involution.
Then either $\Z(A)=F$ or
$\Z(A)$ is a quadratic \'etale extension of $F$ with nontrivial automorphism $\s|_{\Z(A)}$.
We say that the involution $\s$ -- or the algebra with involution $(A,\s)$ -- is
$$
\begin{array}{ccl}
\mbox{\emph{of the first kind}} & \mbox{if} & [\Z(A):F]=1,\\
\mbox{\emph{of the second kind}} & \mbox{if} & [\Z(A):F]=2.
\end{array}
$$
If $\Z(A)$ field, then $A$ is a central simple $\Z(A)$-algebra.
If $(A,\s)$ is of the second kind, then either $\Z(A)$ is a quadratic field extension of $F$ or $\Z(A)\simeq F\times F$.

Involutions on central simple algebras are further distinguished into
three types.
Involutions of the first kind are \emph{of orthogonal}
  (resp.~\emph{symplectic}) \emph{type} if after scalar extension to
a splitting field of the underlying algebra they are adjoint to a diagonalisable (in particular symmetric) bilinear form
(resp.~to an alternating bilinear form); see~\cite[\S2]{BOI}.
Involutions of the second kind are said to be \emph{of
  unitary type. Furthermore, if} $\Z(A)\simeq F\times F$, then we call $(A,\s)$ \emph{unitary of
  inner type}.  (The term is motivated by a corresponding notion for
algebraic groups.)  In this case we obtain that
$(A,\s)\simeq (A_0\times A_0^\op,\sw)$ for a central simple
$F$-algebra $A_0$, its opposite algebra $A_0^\op$ and the \emph{switch
  involution} $\sw\colon A_0\times A_0^\op\to A_0\times A_0^\op$, given by
$\sw(a_1,a_2)=(a_2,a_1)$ (see \cite[(2.14)]{BOI}).  In this situation
we set $\deg A=\deg A_0$, $\ind A=\ind A_0$ and $\coind A=\coind A_0$.

\begin{rems}
  $(1)$\, We have $\Symd(\s)\subseteq \Symm(\s)$ and this is an
  equality unless $\car F = 2$ and $(A,\s)$ is of the first kind. (See
  \cite[(2.17)]{BOI} for the case where $\car F =2$ and $\sigma$
  unitary.)

  $(2)$\, We have $1\notin\Symd(\s)$ if and only if $\s$ is orthogonal
  and $\car(F)= 2$.  Orthogonal involutions in characteristic two are
  peculiar.  When we need to exclude this case, we will simply assume that
  $1\in\Symd(\s)$.
\end{rems}

\begin{prop}\label{P:Syms-dim}
  Set $d=\deg A$, hence $\dim_F A=[\Z(A):F]\cdot d^2$.
If
$1\in\Symd(\s)$, then
$$\dim_F\Symd(\s)\,\,=\,\,
\begin{cases} 
\frac{d(d+1)}2 & \mbox{if $\s$ is orthogonal},\\
\,\,\,\,\, d^2 & \mbox{if $\s$ is unitary},\\
\frac{d(d-1)}2 & \mbox{if $\s$ is symplectic}.
\end{cases}
$$
If $1\notin\Symd(\s)$, then $\dim_F\Symd(\s)=\frac{d(d-1)}2$.
\end{prop}
\begin{proof}
This follows from the definitions together with \cite[(2.6)]{BOI} or
\cite[(2.17)]{BOI}, depending on whether $\s$ is of the first or
second kind. 
\end{proof}

We recall the most basic examples of involutions.
\begin{exs}\label{E:can}\quad
\begin{enumerate}[$(1)$]
\item The identity map $\id_F$ is the unique orthogonal involution on
  $F$, viewed as a central simple $F$-algebra. 
\item Consider a quadratic \'etale extension $K/F$ and  let
  $\can_{K/F}$ denote the nontrivial $F$-automor\-phism of $K$. Then
  $(K,\can_{K/F})$ is an $F$-algebra with unitary involution. 
\item Let $Q$ be an $F$-quaternion algebra.
The unique symplectic involution on $Q$ is given by $x\mapsto
\Trd_{Q}(x)-x$, where $\Trd_Q\colon Q\to F$ denotes the reduced trace form of
$Q$. 
We denote this involution by $\can_Q$ and call it the \emph{canonical
  involution on $Q$}.
\end{enumerate}
\end{exs}
By an \emph{$F$-algebra with canonical involution} we mean an $F$-algebra with involution of one of the three types in \Cref{E:can}.

\section{Capacity} 
\label{sec:cap}

Let $(A,\s)$ be an $F$-algebra with involution.  Following
\cite[Section~5]{BGBT18a}, we call an $F$-subalgebra $L$ of $A$
\emph{neat in $(A,\s)$} or a \emph{neat subalgebra of $(A,\sigma)$} if
$L$ is an \'etale $F$-algebra contained in $\Symm(\sigma)$ and such
that $A$~is free as a left $L$-module and, for all primitive
idempotents $e$ of $L$, the $F$-algebras with involution
$(eAe,\sigma\rvert_{eAe})$ have the same degree and the same type;
this type coincides with the type of $\sigma$.
A subalgebra $B$ of $A$ is called \emph{$\s$-stable} if $\s(B)=B$.

\begin{rem}
  In this article we mostly
  avoid the case of orthogonal involutions
  in characteristic~$2$, which is the most troublesome in the study of
  stable subalgebras, as demonstrated in~\cite{BGBT18a}. In
  particular, in the cases which we consider the definition of neat
  subalgebra simplifies as follows: A commutative $F$-subalgebra $L$
  of $A$ is \emph{neat in} $(A,\sigma)$ if $L$ is \'etale, consists of
  $\sigma$-symmetric elements, and $A$ is free as a left (or right)
  $L$-module.  In particular it follows under these circumstances that, if $L$ is neat in
  $(A,\sigma)$ and $L$ is a free module over some subalgebra $L'$,
  then $L'$ is also neat in $(A,\sigma)$.
\end{rem}

Following \cite{BGBT18a}, we define
\[
\kap(A,\sigma)=\left\{\begin{array}{rl}
\deg A & \mbox{if $\sigma$ is orthogonal or unitary,}\\
\frac{1}{2}\deg A & \mbox{if $\sigma$ is symplectic,}
\end{array}\right.
\]
and we call this integer the \emph{capacity of $(A,\sigma)$}.
\medbreak

Algebras with involution of capacity~$1$ correspond to those in \Cref{E:can}:

\begin{prop}\label{P:can-inv}
The following are equivalent:
\begin{enumerate}[$(i)$]
\item $(A,\s)$ is an $F$-algebra with canonical involution.
\item $\Symd(\s)\subseteq F$.
\item $\kap(A,\s)=1$.
\end{enumerate}
\end{prop}
\begin{proof}
This follows from \Cref{P:Syms-dim}.
\end{proof}

By \cite[Theorem~4.1]{BGBT18a} we have  
\[
  \kap(A,\sigma) = \max\{ [L:F] \mid \text{$L$ \'etale $F$-algebra with
    $L\subseteq \Symm(\sigma)$}\}.
  \]
Furthermore, by \cite[Proposition~5.6]{BGBT18a}, every \'etale
$F$-subalgebra $L$ of $A$ contained in $\Symm(\s)$ and with
$[L:F]=\kap(A,\s)$ is neat in $(A,\s)$.
These commutative subalgebras of $A$ are the maximal ones contained in $\Symd(\s)$.

\begin{prop}
    \label{prop:maximal}
    Let $L$ be an \'etale $F$-subalgebra of $A$ contained in
    $\Symm(\sigma)$ with $[L:F]=\kap(A,\s)$. Then
    $\C_A(L)\cap\Symd(\s)\subseteq L$.
  \end{prop}

\begin{proof}
    By \cite[Proposition~5.6]{BGBT18a}, $L$ is neat in $(A,\s)$.
    Assume first that $\s$ is orthogonal.
    Then $[L:F]=\kap(A,\s)=\deg A$ and, by \cite[Proposition~2.3]{BGBT18a}, we have $\dim_FA=[\C_A(L):F]\cdot[L:F]$. Hence $[\C_A(L):F]=[L:F]$, whereby $\C_A(L)=L$.

  Assume next that $\s$ is unitary. Upon extending scalars we may assume that
    $(A,\s)$ is unitary of inner type. We can identify $(A,\s)$ with
    $(A_0\times A_0^\op,\sw)$ for some central simple $F$-algebra
    $A_0$. Then there is an \'etale $F$-subalgebra $L_0\subseteq A_0$
    such that
    \[
      L=\{(\ell,\,\ell)\mid\ell\in L_0\}.
    \]
    Since $[L:F]=\kap(A,\sigma)$, we have $[L_0:F]=\deg
    A_0$. 
    It follows by~\cite[Proposition~2.3]{BGBT18a} that $\C_{A_0}(L_0)=L_0$,
    hence $L=\C_A(L)\cap\Symd(\s)$.

Assume finally that $\s$ is symplectic. Then $[L:F]=\kap(A,\s)=\frac12\deg A$, and it follows that
    $[\C_A(L):F]=4[L:F]$. 
    By~\cite[Proposition~5.4]{BGBT18a}, all the simple components of
    $\C_A(L)$ have the same degree, and by~\cite[Proposition~3.1]{BGBT18a}, the restriction of $\s$ to $eAe$ is symplectic for every
    nonzero symmetric idempotent $e\in A$.
    Extending scalars, we may
    assume $L$ is split. 
    Let $e_1, \ldots, e_r$ be the distinct primitive
    idempotents of $L$. So $L=e_1F\oplus\cdots\oplus e_rF$ and
    $r=[L:F]=\frac12\deg A$. Then
    \[
      \C_A(L)=(e_1Ae_1)\oplus\cdots\oplus(e_rAe_r)
    \]
    where $e_iAe_i$ is an $e_iF$-quaternion algebra for $1\leq i\leq r$ on which 
     $\s$ restricts to the canonical involution. 
     Consider
    $x\in\C_A(L)\cap\Symd(\s)$, and write $x=y+\s(y)$ for some $y\in
    A$. Since $e_1+\cdots+e_r=1$ and $x$ commutes with $e_1,\dots,e_r$, we
    have
    \[
      x=e_1x+\cdots+e_rx = e_1xe_1+\cdots+e_rxe_r.
    \]
    Now, for $1\leq i\leq r$, we have
    \[
      e_ixe_i=e_iye_i+\s(e_iye_i)\in\Symd(\s\rvert_{e_iAe_i}),
    \]
    and further $\Symd(\s\rvert_{e_iAe_i})=e_iF$, because $\s\rvert_{e_iAe_i}$
    is the canonical involution. Therefore $x\in L$.
    This shows that $\C_A(L)\cap\Symd(\s)\subseteq L$.
  \end{proof}

The following proposition will be needed in Section~\ref{S:dec}
  to prove that \'etale subalgebras along which an algebra with
  involution decomposes are neat.

\begin{prop}\label{P:neat-factor}
Let $L$ be an \'etale $F$-subalgebra of $A$ contained in
$\Symm(\s)$. Assume that there exists a $\s$-stable central simple
$F$-subalgebra $B$ of $A$ such that $L\subseteq B$ and $\C_B(L)=L$. 
Then $L$ is neat in $(A,\s)$.
\end{prop}

\begin{proof}
Let $C=\C_A(B)$, which is a simple $\s$-stable $F$-subalgebra of $A$.
We set $\s_B=\s|_B$ and $\s_C=\s|_C$. 
Then $(B,\s_B)$ and $(C,\s_C)$ are $F$-algebras with involution 
such that $(A,\s)=(B,\s_B)\otimes (C,\s_C)$.
As $\C_B(L)=L$, all simple components of $C_B(L)$ have degree $1$, so we obtain  by~\cite[Proposition~2.3]{BGBT18a} that
$[L:F]=\deg B$. Hence $\kap(B,\s_B)=\deg B$,
and it follows that $\s_B$ is orthogonal.
It follows by \cite[(2.23)]{BOI} that $\s_C$ is of the same type as
$\s$, and consequently 
$\kap(A,\s)=\kap (B,\s_B)\cdot \kap(C,\s_C)$.

We choose an \'etale $F$-subalgebra $M$ of $(C,\s_C)$
contained in
$\Symm(\s_C)$ with $[M:F]=\kap(C,\s_C)$. 
Then $LM$ is an \'etale extension of $F$ contained in $\Symm(\s)$
with
$$
[LM:F]=[L:F]\cdot [M:F]=\kap (B,\s_B)\cdot
\kap(C,\s_C)=\kap(A,\s)\,.
$$ 
Hence $LM$ is neat in $(A,\s)$, by
\cite[Proposition~5.6]{BGBT18a}. Since $LM$ is free as an $L$-module,
it follows by \cite[Lemma~5.8]{BGBT18a} that $L$ is neat in $(A,\s)$. 
\end{proof}

\section{Decomposability} 
\label{S:dec}

In this section we discuss decompositions of algebras with involution that are compatible with a given multiquadratic \'etale subalgebra. We further recall the notion of an algebra with involution being totally decomposable.
\smallskip

Unless explicitly mentioned otherwise, all tensor products of algebras and
vector spaces are taken over $F$ and simply denoted by $\otimes$. 
\smallskip

Let $(A,\s)$ be an $F$-algebra with involution.  We call $(A,\s)$
\emph{totally decomposable} if, for some $n\in\nat$, there exist
$F$-quaternion algebras
$Q_1,\dots,Q_n$ with respective involutions of the first kind $\s_1,\dots,\s_n$ such that
$$(A,\s)\simeq (\Z(A),\s|_{\Z(A)})\otimes\bigotimes_{i=1}^n (Q_i,\s_i)\,;$$
if $\Z(A)=F$, then this simply means that
$(A,\s)\simeq\bigotimes_{i=1}^n (Q_i,\s_i)$.
Recall
from~\cite[(2.22)]{BOI} that every quaternion algebra with unitary
involution is totally decomposable. Therefore, an $F$-algebra with
unitary involution $(A,\s)$ is totally decomposable if and only if it
has a decomposition
\[
  (A,\s) \simeq
  (H_1,\s_1)\otimes_{\Z(A)}\cdots\otimes_{\Z(A)}(H_n,\s_n)
\]
for some quaternion $\Z(A)$-algebras $H_1$, \ldots, $H_n$ with respective $F$-linear unitary
involutions $\s_1$, \ldots, $\s_n$.
The degree of any
totally decomposable $F$-algebra with involution is a power of $2$.

Let $r\in\nat$ and let $B_1,\dots,B_r$ be central simple
$F$-subalgebras of $A$. We call $B_1$, \ldots, $B_r$
\emph{independent} if the $F$-subalgebra of $A$ generated by
$B_1\cup\ldots\cup B_r$, is $F$-isomorphic to the tensor product
$B_1\otimes\cdots \otimes B_r$; this is equivalent to having that
$B_1$, \ldots, $B_r$ are centralizing one another, that is, for any
$i,j\in\{1,\dots,n\}$ with $i\neq j$, we have $xy=yx$ for all
$x\in B_i$ and $y\in B_j$; in this situation we denote the $F$-subalgebra of $A$ generated by $B_1\cup\ldots\cup B_r$  by $B_1\cdots B_r$.

\begin{prop}\label{P:cap2-ttdec}
Let $n\in\nat$ be such that $\kap(A,\s)=2^n$.
Then the following are equivalent:
\begin{enumerate}[$(i)$]
\item $(A,\s)$ is totally decomposable.
\item  There exist $n$ independent $\s$-stable $F$-quaternion subalgebras of $A$.
\item There exist independent $\s$-stable $F$-quaternion subalgebras
  $Q_1,\dots,Q_{n-1}$ of $A$ such that $\s|_{Q_i}$ is orthogonal for
  $1\leq i\leq n-1$. 
\item There exist independent $\s$-stable $F$-quaternion algebras
  $Q_1,\dots,Q_{n}$ of $A$ such that $\s|_{Q_i}$ is orthogonal for
  $1\leq i\leq n$. 
\end{enumerate}
\end{prop}

\begin{proof}
The implication $(i)\Rightarrow (ii)$ is trivial.

\follows{ii}{iii}
This implication follows by induction on $n$, starting with the
trivial cases where $n\leq 1$. 
For the induction step, it suffices to observe that a tensor product
of two $F$-quaternion algebras with
symplectic involutions can up to
isomorphism be rearranged to have an orthogonal involution on at least
one of the two factors.  See e.g.~\cite[Proposition 5.5]{BD16} for a
proof in arbitrary characteristic for this fact. 

\follows{iii}{iv} Assume that  $Q_1$, \ldots, $Q_{n-1}$ are independent $\s$-stable $F$-subalgebras of $A$ such that $\s|_{Q_i}$ is orthogonal for $1\leq i\leq n-1$.
Let $B=Q_1\cdots Q_{n-1}$ and $C=\C_A(B)$.
We set $\s_B=\s|_B$, $\s_C=\s|_C$ and $\s_i=\s|_{Q_i}$ for $1\leq i\leq n-1$.
Then $(C,\s_C)$ is an $F$-algebra with involution, and we have
$$(B,\s_B)\simeq \bigotimes_{i=1}^{n-1}(Q_i,\s_{i})\quad\mbox{ and }\quad (A,\s)\simeq (B,\s_B)\otimes (C,\s_C).$$
Since $\s_1$, \ldots, $\s_{n-1}$ are orthogonal, so is $\s_B$, by
\cite[(2.23)]{BOI}.  
Hence the types of the involutions $\s$ and $\s_C$ coincide, and we have that $\kap(B,\s_B)=2^{n-1}$ and $\kap(A,\s)=\kap(B,\s_B)\cdot \kap(C,\s_C)$.
Since $\kap(A,\s)=2^n=2\cdot \kap(B,\s_B)$, we obtain that $\kap(C,\s_C)=2$. 
Hence there exists a quadratic \'etale  extension $K$ of $F$ contained in $\Symm(\s_C)$, and
by \cite[Corollary 6.6]{BGBT18a}, $K$ is contained in a $\s$-stable
$F$-quaternion subalgebra $Q_n$ of $C$.
Then $Q_1$, \ldots, $Q_n$ are independent and $\s|_{Q_n}$ is
orthogonal. 

\follows{iv}{i}
Assume that  $Q_1$, \ldots, $Q_{n}$ are independent $\s$-stable
$F$-subalgebras of $A$ such that $\s|_{Q_i}$ is orthogonal for $1\leq
i\leq n$. 
We set $B=Q_1\cdots Q_{n}$, $C=\C_A(B)$, $\s_B=\s|_B$, $\s_C=\s|_C$ and $\s_i=\s|_{Q_i}$ for $1\leq i\leq n$.
Then $$(B,\s_B)\simeq \bigotimes_{i=1}^{n}(Q_i,\s_{i})\quad\mbox{ and }\quad(A,\s)\simeq (B,\s_B)\otimes (C,\s_C).$$
Since $\s_1$, \ldots, $\s_{n}$ are orthogonal, $\s_B$ is orthogonal as
well, by \cite[(2.23)]{BOI}.  
We conclude that $\kap(B,\s_B)=2^{n}=\kap(A,\s)$ and that $(C,\s_C)$ is an $F$-algebra with involution with $\kap(C,\s_C)=1$, thus an algebra with canonical involution, by \Cref{P:can-inv}.
If $\s$ is symplectic, then $C$ is an $F$-quaternion algebra, and otherwise $C=\Z(A)$.
Therefore $(A,\s)$ is totally decomposable.
\end{proof}

We retrieve the following well-known fact, which is trivial in the
orthogonal case, and which is contained in \cite[(2.22)]{BOI} in the
unitary case, and in \cite[(16.16)]{BOI} in the symplectic case.

\begin{cor}\label{C:kap2totdec}
If $\kap(A,\s)=2$, then $(A,\s)$ is totally decomposable.
\end{cor}
\begin{proof}
This is the implication $(iii\Rightarrow i)$ in \Cref{P:cap2-ttdec} for $n=1$. (Alternatively the statement is obtained directly from \cite[Corollary 6.6]{BGBT18a}, which was also used in the proof of $(iii\Rightarrow iv)$ in \Cref{P:cap2-ttdec}.)
\end{proof}

In order to investigate decomposability of $(A,\s)$, we first try to
establish the existence of a neat subalgebra and then
to obtain criteria for the decomposability of $(A,\s)$ along this
subextension -- in a sense to be defined. 

Let $L$ be an \'etale $F$-subalgebra of $A$ contained in $\Symm(\s)$. 
We say that $(A,\s)$ \emph{decomposes along $L$} or is \emph{decomposable along $L$} if we have $[L:F]=2^r$ for some $r\in\nat$ and there exist independent $\s$-stable $F$-quaternion subalgebras $Q_1,\dots,Q_r$ of $A$ such that $Q_i\cap L$ is a quadratic $F$-algebra for $1\leq i\leq r$.
Note that in this case $L$ is necessarily neat in $(A,\s)$, by \Cref{P:neat-factor}, and the involution $\s|_{Q_i}$ on $Q_i$ is orthogonal for $1\leq i\leq r$, because $L\cap Q_i$ is an \'etale $F$-algebra contained in $\Symm(\s)$ and  $[L\cap Q_i:F]=2=\deg Q_i$.

\begin{cor}\label{C:cap2-ttdec}
For $n\in\nat$ such that $\kap(A,\s)=2^n$,
the following are equivalent:
\begin{enumerate}[$(i)$]
\item $(A,\s)$ is totally decomposable.
\item  $(A,\s)$ decomposes along some neat $F$-subalgebra $L$ with $[L:F]=2^{n-1}$.
\item $(A,\s)$ decomposes along some neat $F$-subalgebra $M$ with $[M:F]=2^n$.
\end{enumerate}
\end{cor}
\begin{proof}
This is immediate from \Cref{P:cap2-ttdec}.
\end{proof}

If a  neat biquadratic $F$-subalgebra $M$ of $(A,\s)$ along
which $(A,\s)$ 
decomposes is given as a tensor product of quadratic $F$-subalgebras,
one can also find a decomposition of $(A,\s)$ according to the
quadratic subalgebras.

\begin{prop}\label{P:arrange-decompalong}
Let $L$ be a neat biquadratic $F$-subalgebra of $(A,\s)$
along which $(A,\s)$ decomposes. Let $K_1$, $K_2$ 
be quadratic \'etale $F$-subalgebras of $A$ such that $L=K_1K_2$
and such that $L$ is free as a $K_i$-module for $i=1$, $2$. 
Then there exist independent $\s$-stable $F$-quaternion subalgebras
$Q_1$, $Q_2$ of $A$ such that $Q_i\cap L=K_i$ for $i=1$, $2$.
\end{prop}
\begin{proof}
By the hypothesis, we may assume that there exist two independent
$\s$-stable $F$-quaternion subalgebras $Q$ and $Q'$ of $A$ such that
$K=Q\cap L$ and $K'=Q'\cap L$ are quadratic $F$-subalgebras of $L$. 
Note that $K\neq K'$. 
Apart from $K_1$ and $K_2$, there exists precisely one further quadratic \'etale $F$-subalgebra $K_3$ of $A$ over which $L$ is free as a module. It follows that $K,K'\in\{K_1,K_2,K_3\}$. 
Up to switching the roles of $Q$ and $Q'$, we can therefore assume  without
loss of generality that $K=K_1$ and $K'\in \{K_2,K_3\}$. If
 $K'=K_2$, then we set $Q_1=Q$ and $Q_2=Q'$ and are done. 
Assume now that $K'=K_3$.
Note that the involutions $\s|_Q$ and $\s|_{Q'}$ are orthogonal.
We may take $j\in \mg{Q}\cap\Symm(\s)$ and $j'\in
\mg{Q'}\cap\Symm(\s)$ such that 
$\Int(j)|_{K}$ and $\Int(j')|_{K'}$ are the nontrivial
$F$-automorphisms of $K$ and $K'$, respectively. 
Set $j''=jj'$.  
It follows that $K_2=\{x\in L\mid j''x=xj''\}$ and further that
$Q_1=K_1\oplus j''K_1$ and $Q_2= K_2\oplus j'K_2$ are $\s$-stable
independent $F$-quaternion subalgebras of $A$ such that $Q\cdot
Q'=Q_1\cdot Q_2$.
\end{proof}

\section{Hyperbolic involutions and quaternion factors} 
\label{S:hyperquatfactor}

Let $(A,\s)$ be an $F$-algebra with involution.
We call the involution $\s$ \emph{isotropic} if $\s(x)x=0$ for some
$x\in A\setminus\{0\}$, and \emph{anisotropic} otherwise.  
We call $\s$  \emph{metabolic}  if there exists $e\in A$ such that
$e^2=e$, $\sigma(e)e=0$ and $\dim_FeA=\frac{1}{2}\dim_F A$.  
We call $\s$  \emph{hyperbolic} if there exists $e\in A$ such
that $e^2=e$, $\sigma(e)=1-e$.
Note that every metabolic involution is isotropic. 
Note further that every symplectic involution on a split algebra and
every unitary involution of inner type is hyperbolic.
We recollect from \cite[Lemma~A.3]{BFT} and \cite[Proposition
4.10]{Dol11} the following fact.

\begin{prop}\label{P:metahyp}
The involution $\s$ is hyperbolic if and only if $\s$ is metabolic and
$1\in\Symd(\s)$. 
\end{prop}

\begin{proof}
Assume that $\s$ is hyperbolic. 
Fix $e\in A$ such that $e^2=e$ and $\s(e)=1-e$. Then
$1=e+\s(e)\in\Symd(\s)$ and $\s(e)e=0$. 
Furthermore $A=eA\oplus (1-e)A$  and $Ae=\s(Ae)=\s(e)A=(1-e)A$, whereby
$\dim_F eA = \dim_F Ae=\dim_F (1-e)A$ and $\dim_F eA=\frac 12 \dim_F A$. Hence $\s$ is
metabolic.

Assume now that $\s$ is metabolic and $1\in\Symd(\s)$. 
If $\car(F)\neq 2$, then it follows by \cite[Proposition 4.10]{Dol11}
that $\s$ is hyperbolic. 
If $\car(F)=2$, then the condition that $1\in\Symd(\s)$ says that $\s$
is not orthogonal, and it follows by \cite[Lemma~A.3]{BFT} that $\s$
is hyperbolic. 
\end{proof}

We are going to characterise the hyperbolicity of the involution $\s$
by the existence of certain $\s$-stable $F$-quaternion subalgebras. 
The following statement provides the basis to this approach.

\begin{prop}\label{P:kap2isometa}
If $\kap(A,\s)=2$, then $\s$ is either anisotropic or metabolic.
\end{prop}

\begin{proof}
If either $(A,\s)$ is unitary of inner type, or 
$A$ is split and $\s$ is symplectic,
then  $\s$ is hyperbolic and hence metabolic. 
Assume now that we are in neither of these two cases and $\kap(A,\s)=2$. 
Then every nontrivial right ideal of $A$ has $F$-dimension equal to $\frac{1}2\dim_FA$. 
Suppose that $\s$ is isotropic. Fix $x\in A\setminus\{0\}$ such that 
$\s(x)x=0$. Then $xA=eA$ for some $e\in A$ with $e^2=e$, and we obtain that
$\s(e)e=0$. Furthermore $eA$
is a nontrivial right ideal of $A$, whereby $\dim_FeA=\frac{1}2\dim_FA$. 
Hence $\s$ is metabolic.
\end{proof}

The following statement is a variation of \cite[Theorem 2.2]{BST},
without restriction on the characteristic. 

\begin{prop}\label{L:meta-quat-factor}
Assume that $\kap(A,\s)$ is even and $1\in\Symd(\s)$.
Then the following are equivalent:
\begin{enumerate}[$(i)$]
\item
$\s$ is hyperbolic and $\coind A$ is even.
\item There exists a split $\s$-stable $F$-quaternion subalgebra $Q\subseteq A$ such that $\s|_Q$ is orthogonal and metabolic.
\item There exist elements
$u,v\in A$ such that 
$u^2=u$, $v^2=1$, $uv+vu=v$, $\s(u)=1-u+uv$ and $\s(v)=-1+2u+v-uv$.
\end{enumerate}
\end{prop}

Note that, if $\s$ is hyperbolic and $A$ is simple, then $\coind A$ is necessarily even, so the second condition
in~(\emph{i}) is relevant only in the case where $(A,\s)$ is unitary of inner type.
Also, each of the
properties~(\emph{i}), (\emph{ii}), (\emph{iii}) implies that
$\deg A$ is even. Hence, the condition that $\kap(A,\s)$ is even
is needed only when $\s$ is symplectic.
\medskip

\begin{proof}
\follows{iii}{ii}
Assume that $u$, $v\in A$ satisfy the relations in $(iii)$.
Then $Q=F\oplus uF\oplus vF\oplus uvF$ is a $\s$-stable $F$-quaternion
subalgebra of $A$. 
Note that $\s|_Q$ is not the canonical involution on $Q$, because
$u^2=u$ and $\s(u)\neq 1-u$. 
Hence $\s|_Q$ is orthogonal.
Since $$\s(u)u=(1-u+uv)u=uvu=u(v-uv)=(u-u^2)v=0\,,$$ we have that
$\s|_Q$ is isotropic. 
As $Q$ is an $F$-quaternion algebra, it follows by
\Cref{P:kap2isometa} that $\s|_Q$ is metabolic. 
Moreover, as $\s(u)u=0$, $Q$ is not a division algebra, and since it is a quaternion algebra, it is therefore split.

\follows{ii}{i}
Let $Q$ be a split $\s$-stable $F$-quaternion subalgebra of $A$ such that $\s|_Q$ is orthogonal and metabolic.
Let $C$ be the centralizer of $Q$ in $A$.
Then $C$ is $\s$-stable and 
$$(A,\s)\simeq (Q,\s|_Q)\otimes (C,\s|_C)\,.$$
Since $Q$ is split and $\s|_Q$ is metabolic, it follows that $\coind A$ is even and $\s$ is metabolic.
Hence $1\in\Symd(\s)$, and we obtain by \Cref{P:metahyp} that $\s$ is hyperbolic.

\follows{i}{iii}
Suppose that $(A,\s)$ is hyperbolic and $\coind A$ is even.
Since $\kap(A,\s)$ is even as well, there exists an $F$-algebra with involution $(B,\tau)$ of the same type as $(A,\s)$ and such that $A\simeq \matr{2}(F)\otimes B$.
The matrices 
$u=\left(\begin{smallmatrix} 0 & 0 \\ 0 & 1\end{smallmatrix}\right)$ and $v=\left(\begin{smallmatrix} 0 & 1 \\  1 & 0\end{smallmatrix}\right)$
in $\matr{2}(F)$ satisfy the relations in $(iii)$ with respect to the involution $\s'=\Int(u+v)\circ t$ on $\matr{2}(F)$, where $t$ denotes the transposition involution on $\matr{2}(F)$. 
As in the proof of \follows{iii}{ii} we obtain that $(\matr{2}(F),\s')$ is metabolic.
It follows that $(\matr{2}(F),\s')\otimes (B,\tau)$ is metabolic.
Since $\s'\otimes \tau$ is of the same type as $\s$, we have that $1\in\Symd(\s'\otimes \tau)$ and
conclude by \Cref{P:metahyp} that $(\matr{2}(F),\s')\otimes (B,\tau)$ is hyperbolic.
It follows from \cite[(12.35)]{BOI} that all algebras with hyperbolic involution of the same type and with the same underlying algebra are isomorphic.
Since $A\simeq \matr{2}(F)\otimes B$, we conclude that $(A,\s)\simeq (\matr{2}(F),\s')\otimes (B,\tau)$.
Hence $A$ contains elements $u$ and $v$ satisfying the equations in $(iii)$ with respect to $\s$.
\end{proof}

\begin{cor}\label{C:meta-quat-descent}
Let $K$ be a separable field extension of $F$ contained in $\Symm(\s)$.
Let $C=\C_A(K)$ and
assume that $\coind C$ and $\kap(C,\s|_C)$ are even and $\s|_C$ is hyperbolic.
Then there exists a $\s$-stable split $F$-quaternion subalgebra $Q$ of $A$ such that $\s|_Q$ is orthogonal and metabolic and such that $K\subseteq \C_A(Q)$.
\end{cor}

\begin{proof}
Since $\s|_C$ is hyperbolic, we have $1\in\Symd(\s)$. 
By \Cref{L:meta-quat-factor}, there exist elements $u,v\in C$ such that 
$u^2=u$, $v^2=1$, $uv+vu=v$, $\s(u)=1-u+uv$ and $\s(v)=-1+2u+v-uv$.
Then $Q=F\oplus Fu\oplus Fv\oplus Fuv$ is a $\s$-stable $F$-quaternion subalgebra of $A$ such that $\s|_Q$ is orthogonal and metabolic and such that $K\subseteq \C_A(Q)$.
\end{proof}

Our last two results in this section provide conditions under which a
quadratic \'etale $F$-algebra $K$ contained in $\Symm(\s)$ can be
embedded into a split $\sigma$-stable $F$-quaternion subalgebra $Q$
such that $\s\rvert_Q$ is orthogonal and metabolic. We  consider
in \Cref{P:hyp-neat-in-quat-meta-factor} the case where $K$ is split,
and then in \Cref{thm:existmetabolic} the case where $K$ is a field.
In the following proof we interpret an involution as adjoint to a
hermitian form and use the correspondence between the concepts of
hyperbolicity for the two sorts of objects. 

\begin{prop}\label{P:hyp-neat-in-quat-meta-factor}
Let $K$ be a neat subalgebra of $(A,\s)$ such that $K\simeq F\times
F$. Assume that $\s$ is hyperbolic. 
Then $K$ is contained in a split $\s$-stable $F$-quaternion subalgebra
$Q$ of $A$ such that $\s|_Q$ is orthogonal and metabolic. 
\end{prop}
\begin{proof}
We first assume that $A$ is simple.
 In this case we may identify $A$ with $\End_D(V)$ for a
 finite-dimensional $F$-division algebra $D$ and a finite-dimensional
 $D$-right vector space $V$. 
We fix an involution $\tau$ on $D$ of the same kind as $\s$.
By \Cref{P:metahyp}, since $\s$ is hyperbolic, $\s$ is not orthogonal
if $\car F= 2$. Note further that, if $\car F\neq 2$ or if $\tau$ is
unitary, then every hermitian or skew-hermitian form over $(D,\tau)$
is even, by \cite[Chap.~I, Lemma 6.6.1]{Knus}. 
Using the identification of $A$ with $\End_D(V)$, we therefore obtain
by \cite[(4.2)]{BOI} that the involution $\s$ is adjoint to a
nondegenerate even hermitian or skew-hermitian form $h\colon V\times
V\to D$ with respect to $\tau$. 

Let $e$ and $e'$ be the primitive idempotents in $K$. 
Hence $e'=1-e$ and  we have that $e$ and $e'$ are the two projections
given by a direct decomposition $V=W\oplus W'$ for two $D$-subspaces
$W$ and $W'$ of $V$. Since $e,e'\in K\subseteq \Symm(\s)$ we obtain
that $W$ and $W'$ are orthogonal to one another with respect to $h$. 
Hence, we have an orthogonal decomposition
$$(V,h)\simeq (W,h|_{W})\perp (W',h|_{W'})\,.$$
Since $K$ is neat in $(A,\s)$, we have that $(eAe,\s|_{eAe})$ and
$(e'Ae',\s|_{e'Ae'})$ have the same degree and the same type. 
Since these $F$-algebras with involution correspond to the
$D$-endomorphism algebras of $W$ and of $W'$ with their involutions
adjoint to the restrictions of $h$, we conclude that  
$\dim_DW=\dim_DW'=\frac{1}2\dim_DV$ and that the restrictions of $h$
to $W$ and $W'$ have the same type as $h$, hermitian or
skew-hermitian. 

Since $\s$ is hyperbolic, so is $(V,h)$. 
Since also $(W',-h|_{W'})\perp (W',h|_{W'})$ is hyperbolic and of the
same dimension as $(V,h)$, it follows by  
\cite[Proposition~7.7.3]{Scharlau} that $(V,h)\simeq (W',-h|_{W'})\perp (W',h|_{W'})$.
Since $h$ is even, so is $h|_{W'}$, and hence we may apply Witt Cancellation
\cite[Chap.~I, Proposition 6.4.5]{Knus} and conclude that
$$(W,h|_{W})\simeq -(W',h|_{W'})\,.$$
Let $g\colon W\to W'$ be an isometry between $h|_W$ and $-h|_{W'}$.
Let $f\colon V\to V$ be the $D$-automorphism of $V$ determined by 
$f(w+ w')=g^{-1}(w')+ g(w)$ for $w\in W$ and $w'\in W'$.
It follows that $\s(f)=-f$, $ef+fe=f$ and $f^2=\id_V$.
We conclude that $e$ and $f$ generate a split $F$-quaternion
subalgebra $Q$ which is $\s$-stable and such that $\s|_Q$ is
orthogonal.  
Furthermore, since $f\neq 1$ and $$\s(1-f)\cdot
(1-f)=(1+f)(1-f)=1-f^2=0,$$ we conclude that $\s|_Q$ is
isotropic. Since $Q$ is an $F$-quaternion algebra, it follows by
\Cref{P:kap2isometa} that $\s|_Q$ is metabolic. This concludes the
proof for the case where $A$ is simple. 

Suppose finally that $(A,\s)$ is unitary of inner type.
Hence, we may identify $(A,\s)$ with $(B\times B^\op,\sw)$ for some
central simple $F$-algebra $B$.  
We obtain that $K=\{(x,x)\mid x\in K_0\}$ for an $F$-algebra $K_0$
contained in $B$, isomorphic to $F\times F$ and such that $B$ is free
as a $K_0$-left module. 
Hence $\coind B$ is even.
We use a variation of the above argument, without involutions and hermitian forms.
We identify $B$ with $\End_D(V)$ for a finite-dimensional $F$-division
algebra $D$ and a finite-dimensional $D$-right vector space $V$. 
As in the previous case the two primitive idempotents of $K_0$ give
rise to a decomposition $V=W\oplus W'$ where the $D$-subspaces $W$ and
$W'$ of $V$ are of the same dimension and therefore isomorphic.  
We fix a $D$-isomorphism $g\colon W\to W'$ and then define a
$D$-automorphism $f$ of $V$ determined by letting 
$f(w+ w')=g^{-1}(w')+ g(w)$ for $w\in W$ and $w'\in W'$.
Let $e\colon V\to W$ be the projection on the first component for the
decomposition  $V=W\oplus W'$. 
Then $f$ and $e$ generate a split $F$-quaternion subalgebra $Q_0$ of $B$.
We fix an orthogonal metabolic involution $\tau$ on $Q_0$.
Then $Q=\{(x,\tau({x}))\mid x\in Q_0\}$ is a split $F$-quaternion
subalgebra of $A=B\times B^\op$ containing $K$ and stable under
$\s=\sw$. 
Furthermore, $(Q,\s|_Q)\simeq (Q_0,\tau)$, whereby $\s|_Q$ is
orthogonal and metabolic. 
\end{proof}

\begin{thm}
  \label{thm:existmetabolic}
 Let $K$ be a separable quadratic field extension of $F$  contained in
 $\Symm(\s)$ and let $C=\C_A(K)$. Assume that $\s$ is hyperbolic and
 that $\sigma|_C$ is anisotropic.  
Then $K$ is 
contained in a $\s$-stable split $F$-quaternion subalgebra $Q$ such
that $\s|_Q$ is orthogonal and metabolic.
\end{thm}

\begin{proof}
Let $\gamma$ be the nontrivial $F$-automorphism of $K$ and $$C'=\{x\in
A\mid xk=\gamma(k)x\mbox{ for every } k\in K\}\,.$$ 
By \cite[Proposition 6.1]{BGBT18a} we have that $\s(C)=C$, $\s(C')=C'$
and $A=C\oplus C'$. 

Let $e\in A$ be such that $e^2=e$ and $\s(e)=1-e$.
We write $e=v+w$ with
  $v\in C$ and $w\in C'$. 
We have
$$  (v^2+w^2) + (vw+wv) =e^2=e=v+w.$$
Note that $v^2$, $w^2\in C$ and $vw,wv\in C'$.
As $A=C\oplus C'$, we obtain that
$$    v^2+w^2=v
\quad\mbox{  and }\quad
    vw+wv=w.$$
Furthermore
$\sigma(v)+\sigma(w)=\sigma(e)=1-e= (1-v)
  - w$.
As $\s(C)=C$ and $\s(C')=C'$, it follows that
  $$    \sigma(v) = 1-v\quad\mbox{ and } \quad    \sigma(w) = -w\,.$$
We conclude that 
$$wv=w-vw=\s(v)w\,.$$
We claim that there exists $v'\in C$ with $vv'=1$.
Suppose $x\in C$ is such that $vx=0$. 
Then  $\sigma(x)\cdot(1-v)=\sigma(vx)=0$, and therefore
  \[\sigma(x)=\sigma(x)v.\]
  It follows that
  \[
  \sigma(x)x=\sigma(x)vx=0\,.
  \]
Since $\sigma|_C$ is anisotropic we conclude that $x=0$. 
Therefore the $F$-linear map $C\to C$ given by multiplication with $v$
from the left is injective. Since $C$ is finite-dimensional, this map
is also surjective. Hence there exists an element $v'\in C$ with
$vv'=1$. 
Using that $wv=\s(v)w$ we obtain that 
$$\s(v')w=\s(v')wvv'=\s(v')\s(v)wv'=\s(vv')wv'=wv'\,.$$
As $\s(w)=-w$ we conclude that $\s(wv')=-wv'$ and 
$$(wv')^2=\s((wv')^2)=\s(v')w\s(v')w=\s(v')w^2v'=\s(v')(v-v^2)v'=\s(v')(1-v)\,.$$
As $\s(v)=1-v$ we obtain that $(wv')^2=1$. 
Since $wv'\in C'$ it follows that 
$Q=K\oplus Kwv'$ is a split $\s$-stable $F$-quaternion subalgebra.
Since $K$ is a quadratic \'etale $F$-algebra and $\s|_K=\id_K$, the involution $\s|_Q$ is orthogonal. 
Since $$\s(1+wv')(1+wv')=(1-wv')(1+wv')=1-(wv')^2 = 0\,,$$
we have that $\s|_Q$ is metabolic.
\end{proof}

\section{Hyperbolicity in capacity four} 
\label{sec:cap4}

Our study of algebras with involution of capacity $4$ in \Cref{sec:biquad} will crucially rely on the special case where the involution is hyperbolic, which we study in this section.
We start by showing that an algebra with hyperbolic involution in capacity four is decomposable along any 
quadratic neat  subalgebra (\Cref{T:hyp-quad-factor}), except in one special case.
We will then show that any decomposable unitary or symplectic involution of capacity four can be made hyperbolic by passing to the function field of some quadratic form of dimension at least five.
This will allow us in \Cref{sec:biquad} to show that a certain Pfister form which we will attach to an algebra with involution of capacity $4$ is independent from certain choices which we make in its construction.
\medskip

Let $(A,\s)$ be an $F$-algebra with involution. 

\begin{prop}
  \label{T:hyp-quad-factor}
  Assume that $\kap(A,\s)=4$, $\exp A\leq 2$ and $\s$ is hyperbolic.
  Then $(A,\s)$ is decomposable along every quadratic neat 
  $F$-subalgebra of $(A,\s)$.
\end{prop}

\begin{proof}
We have $1\in\Symd(\s)$, because $\s$ is hyperbolic.
Let $K$ be an arbitrary quadratic neat  $F$-subalgebra of $(A,\s)$.
To prove the statement we need to show that $K$ is contained in a
$\s$-stable $F$-quaternion subalgebra of $A$. 
When $K\simeq F\times F$ this already follows by
\Cref{P:hyp-neat-in-quat-meta-factor}, so we may assume that $K$ is a
field. 

Suppose first that $A$ is not simple. In this case we may identify
$(A,\s)$ with $(B\times B^\op,\sw)$ for some central simple
$F$-algebra $B$. 
Then $K=\{(x,x)\mid x\in K_0\}$ for a separable quadratic field
extension $K_0$ of $F$  contained in $B$. 
Since we have $\exp B=\exp A\leq 2$, it follows by a theorem of
Albert, \cite[(16.2)]{BOI}, that $K_0$ is contained in an
$F$-quaternion subalgebra $Q_0$ of $B$.  
We fix an orthogonal involution $\tau$ on $Q_0$ with $\tau|_{K_0}=\id_{K_0}$.
Then $Q=\{(x,\tau({x}))\mid x\in Q_0\}$ is an $F$-quaternion
subalgebra of $A=B\times B^\op$ containing $K$ and $Q$ is stable under
the involution $\s=\sw$. 
Furthermore, $(Q,\s|_Q)\simeq (Q_0,\tau)$, whereby $\s|_Q$ is orthogonal.

Assume now that $A$ is simple.
Let $C=\C_A(K)$ and $\s_C=\s|_C$. Then $C$ is simple and $(C,\s_C)$ is
a $K$-algebra with involution with
$\deg C=\frac12\deg A$. Since  by~\cite[Proposition~3.3]{BGBT18a}
$\s_C$ has the same type as $\s$,
it follows that $\kap(C,\s_C)=2$. 
If $\s_C$ is anisotropic, then we obtain the desired conclusion by \Cref{thm:existmetabolic}.
Assume now that $\s_C$ is isotropic.
As $C$ is simple and $\s_C$ is isotropic, $\coind C$ is even.
As further $\kap(C,\s_C)=2$, it follows by \Cref{P:kap2isometa} and \Cref{P:metahyp} that $(C,\sigma_C)$ is hyperbolic.
Hence, by \Cref{C:meta-quat-descent},  $C$ contains a $\s$-stable split $F$-quaternion subalgebra $Q'$ such that $(Q',\s|_{Q'})$ is orthogonal and metabolic. Then $D=\C_A(Q')$ is a $\s$-stable $F$-subalgebra of $A$ containing $K$. Let $\s_D=\s|_D$. 
Then $(D,\s_D)$ is an $F$-algebra with involution with $\kap(D,\s_D)=2$ and $\s_D$ is of the same type as $\s$.
Since $K\subseteq \Symm(\s_D)$, it follows by \cite[Corollary 6.6]{BGBT18a} that $K$ is contained in a $\s$-stable $F$-quaternion subalgebra $Q$ of $D$, and hence of $A$.
\end{proof}

The following example explains why the condition on the exponent
cannot be omitted in the statement of \Cref{T:hyp-quad-factor}.

\begin{ex}
  \label{ex:exp4}
  Let $B$ be a central $F$-division algebra with $\exp B=\deg B=4$.
  Then $(B\times B^\op,\sw)$ is an $F$-algebra with unitary involution
  of capacity $4$ which is hyperbolic and indecomposable.
\end{ex}

We will use standard notation from \cite{EKM} for diagonal quadratic
forms in characteristic different from $2$ and for nonsingular binary
quadratic forms in arbitrary characteristic.  We recall some quadratic
form terminology from \cite[(7.17)]{EKM}, in particular concerning the
radicals of a quadratic form and of its polar form.

Let $q:V\to F$ be a quadratic form over $F$, defined on a
finite-dimensional $F$-vector space $V$. 
We denote by $b_q$ the \emph{polar form of~$q$} given by $$V\times
V\to F,\qquad (x,y)\mapsto q(x+y)-q(x)-q(y)\,.$$ 
We further set 
\begin{eqnarray*}\rad(b_q) & = & \{x\in V\mid b_q(x,y)=0\mbox{ for all } y\in V\}\\
\rad(q) & = & \{x\in \rad(b_q)\mid q(x)=0\}\,
\end{eqnarray*}
and observe that these are $F$-subspaces of $V$ with $\rad(q)\subseteq \rad(b_q)$. 
Moreover,
if $\car F\neq 2$ then $q(x)=\frac{1}{2}b_q(x,x)$ for all $x\in V$ and
thus $\rad(q)=\rad(b_q)$. 
We call the quadratic form $q$ \emph{regular} if $\rad(q)=\{0\}$ and
\emph{nondegenerate}  
 if $q$ is regular and $\dim_F\rad(b_q)\leqslant 1$.

An $F$-subspace $V$ of $\Symm(\s)$ such that $x^2\in F$ for all $x\in V$ gives rise to 
a quadratic form $\psi:V\to F, x\mapsto x^2$ and thus yields an $F$-algebra homomorphism $\C(\psi)\to A$ where $\C(\psi)$ is the Clifford algebra of $\psi$, which relates $\s$ to the standard involution on $\C(\psi)$.
This can be used to obtain criteria for decomposability of $(A,\s)$, and it will be used to this purpose in \Cref{L:make-decomp-hyp}.

\begin{lem}\label{L:sim-space-5-6}
Assume that $\kap(A,\s)=4$, $\s$ is unitary or symplectic and $(A,\s)$
is totally decomposable. 
Then there exists an $F$-subspace $V$ of $\Symm(\s)$ with
$$\dim_FV\,=\,
\left\{
\begin{array}{cl} 
6 & \mbox{ if $\s$ is symplectic},\\
5 & \mbox{ if $\s$ is unitary},
\end{array}
\right.$$
such that $x^2\in F$ for all $x\in V$ and such that $\psi\colon V\to F$,
$x\mapsto x^2$ is a nondegenerate quadratic form. 
Furthermore, if for such an $F$-space $V$ the corresponding form $\psi$
is isotropic, then $\s$ is hyperbolic. 
\end{lem}

\begin{proof}
Assume first that $\s$ is symplectic.
Since $(A,\s)$ is totally decomposable, we obtain that 
$$(A,\s)\simeq(Q_1,\can_{Q_1})\otimes (Q_{2},\can_{Q_2})\otimes
(Q_{3},\can_{Q_3})$$ for three $F$-quaternion algebras $Q_1,Q_2,Q_3$. 
For $i=1,2,3$, using the different presentations for quaternion algebras given in~\cite[p.~25]{BOI},
we fix $u_i,v_i\in Q_i$ with
\begin{equation}
\label{EQ:quat}
u_i^2,v_i^2\in\mg{F}\mbox{ and }u_iv_i+v_iu_i=
\left\{\begin{array}{ll} 0 & \mbox{ if }\car(F)\neq 2,\\
1 & \mbox{ if }\car(F)= 2.\end{array}\right.
\end{equation}
Set 
$$
V=
\left\{\begin{array}{cl}
(Fu_1\oplus Fv_1\oplus Fu_1v_1)u_3 \oplus (Fu_2\oplus Fv_2\oplus Fu_2v_2)v_3
& \mbox{if }\car(F)\neq 2,\\ 
Fu_1\oplus Fv_1\oplus Fu_2\oplus Fv_2\oplus Fu_3\oplus Fv_3 & \mbox{if }\car(F)=2.\end{array}\right.$$
Then $\dim_FV=6$ and $V\subseteq \Symm(\s)$, 
and one can easily check that $x\mapsto x^2$ defines a nondegenerate
quadratic form $V\to F$: 
In fact, letting $a_i=u_i^2,b_i=v_i^2\in\mg{F}$ for $i=1,2,3$, we obtain that 
$$
\psi\,\simeq\, 
\left\{\begin{array}{cl}
a_3\la a_1,b_1,-a_1b_1\ra\perp b_3\la a_2,b_2,-a_2b_2\ra & \mbox{if }\car(F)\neq 2,\\ 
{}[a_1,b_1]\perp [a_2,b_2]\perp [a_3,b_3] 
& \mbox{if }\car(F)= 2.
\end{array}\right.$$

Assume now that $\s$ is unitary.
Then $$(A,\s)\simeq (K,\s_{K})\otimes (Q_1,\can_{Q_1})\otimes (Q_{2},\can_{Q_2})$$ for $K=\Z(A)$, $\s_K=\s|_K$ and two $F$-quaternion algebras $Q_1,Q_2$.
If $\car(F)\neq 2$ then let $d\in\mg{F}$ be such that $K\simeq F(\sqrt{d})$. Hence there exists $w\in K$ such that $K=F\oplus Fw$ and $w^2=d$.
For $i=1$, $2$ we fix $u_i$, $v_i\in Q_i$ satisfying (\ref{EQ:quat}).
Set 
$$
V=
\left\{\begin{array}{cl}
(Fu_1\oplus Fv_1\oplus Fu_1v_1)u_2 \oplus Fwv_2\oplus Fwu_2v_2
& \mbox{if }\car(F)\neq 2,\\ 
Fu_1\oplus Fv_1\oplus Fu_2\oplus Fv_2\oplus F & \mbox{if }\car(F)=2.\end{array}\right.$$
Then $\dim_FV=5$ and $V\subseteq\Symm(\s)$, and one can easily check
that $x\mapsto x^2$ defines a nondegenerate quadratic form $\psi\colon
V\to F$:
In fact, letting $a_i=u_i^2,b_i=v_i^2\in\mg{F}$ for $i=1,2$, we obtain that 
$$
\psi\,\simeq\, 
\left\{\begin{array}{cl}
a_2\la a_1,b_1,-a_1b_1\ra\perp db_2\la 1,-a_2\ra & \mbox{if }\car(F)\neq 2,\\ 
{}[a_1,b_1]\perp [a_2,b_2]\perp \la 1\ra
& \mbox{if }\car(F)= 2.
\end{array}\right.$$

After having chosen $V$ and $\psi$ case by case, the rest of the argument
can be given uniformly for all four cases. 
Note that
$$\hspace{3cm}\psi(x+y)-\psi(x)-\psi(y)=xy+yx\qquad\mbox{ for any }\,\, x,y\in V\,.$$

Suppose now that $\psi$ is isotropic.
Fix $x\in V\setminus\{0\}$ such that $\psi(x)=0$.
Since $\psi$ is nondegenerate, there exists $y\in V\setminus\{0\}$
such that $xy+yx=1$. 
It easily follows that $Q=F\oplus Fx\oplus Fy\oplus Fxy$ is a
$\s$-stable $F$-quaternion subalgebra of~$A$. 
Since $x\s(x)=x^2=\psi(x)=0$, 
it follows that $\s|_Q$ isotropic, and since $Q$ is an $F$-quaternion
algebra, we conclude that $\s|_Q$ is metabolic. This implies that $\s$
is metabolic. 
As  $\s$ is symplectic or unitary, we obtain by \Cref{P:metahyp} that
$\s$ is hyperbolic. 
\end{proof}

\begin{lem}\label{L:make-decomp-hyp}
Assume that $(A,\s)$ is totally decomposable, $\kap(A,\s)=4$ and $1\in\Symd(\s)$.
Let $n=1$ if $\s$ is orthogonal, $n=2$ if $\s$ is unitary, and $n=3$
if $\s$ is symplectic. 
There exists a field extension $F'/F$   such that $(A,\s)_{F'}$ is
hyperbolic and every anisotropic quadratic $n$-fold Pfister form over
$F$ remains anisotropic over $F'$. 
\end{lem}
\begin{proof} 
Assume first that $\s$ is orthogonal. Then the hypothesis implies that
$\car(F)\neq 2$. 
It is well-known (see e.g.~\cite[Corollary 15.12]{BOI} or
\cite[Proposition 3.9]{BU18}) that every tensor product of two 
$F$-quaternion algebras with orthogonal involution is isomorphic to a
tensor product of two $F$-quaternion algebras with canonical
involution. 
In particular, since $(A,\s)$ is totally decomposable, there exists a
$\s$-stable $F$-quaternion subalgebra $Q$ of $A$ on which $\s$
restricts to the canonical involution. 
Let $F'/F$ be the function field of the Severi--Brauer variety
associated with $Q$, or equivalently, the function field of the
projective conic given by the pure part of the norm form of $Q$. 
Then $Q_{F'}$ is split, whereby $\s|_{Q_{F'}}$ is hyperbolic.
Hence $(A,\s)_{F'}$ is hyperbolic.
Since $F$ is relatively algebraically closed in $F'$, every
anisotropic $1$-fold Pfister form over $F$ stays anisotropic
over~$F'$. 

Suppose now that $\s$ is unitary or symplectic.
Since $(A,\s)$ is totally decomposable, by \Cref{L:sim-space-5-6}
there exists an $F$-subspace $V$ of $\Symm(\s)$ with 
$$\dim_FV\,=\,
\left\{
\begin{array}{cl} 
6 & \mbox{ if $\s$ is symplectic},\\
5 & \mbox{ if $\s$ is unitary},
\end{array}
\right.$$
such that $x^2\in F$ for all $x\in V$ and such that $V\to F$,
$x\mapsto x^2$ is a nondegenerate quadratic form. 
We denote this quadratic form by $\psi$.
Let $F(\psi)$ be the function field of the projective quadric defined
by $\psi$ over $F$. 
Since $\psi_{F(\psi)}$ is isotropic, it follows by
\Cref{L:sim-space-5-6} that the $F(\psi)$-algebra with involution
$(A,\s)_{F(\psi)}$ is hyperbolic. 

Our first attempt for choosing $F'$ is to take $F(\psi)$.
Suppose that this choice is not satisfying the claim.
Then there must exist an anisotropic quadratic $n$-fold Pfister form
$\pi$ over $F$ such that $\pi_{F(\psi)}$ is isotropic. 
Since quadratic Pfister forms are either anisotropic or hyperbolic, it
follows that $\pi_{F(\psi)}$ is hyperbolic. 
By the Subform Theorem \cite[(22.5)]{EKM}, this implies that $\psi$ is
similar to a subform of $\pi$.  
In particular, $5\leq \dim(\psi)\leq \dim(\pi)$, whereby $n=3$. 
Hence the involution $\s$ is symplectic and $\dim(\psi)=6$.
It follows that $\pi$ is similar to $\psi\perp\beta$ for a regular
$2$-dimensional quadratic form $\beta$ over $F$. 
Hence $\beta$ is similar to the norm form of a separable quadratic  extension $K/F$.
Since $\pi$ is anisotropic over $F$, so is $\beta$, whereby $K$ is a field.
Moreover $\beta_K$ is hyperbolic, and since $\pi_K$ is a Pfister form
containing a form similar to $\beta_K$, it follows that $\pi_K$ is
hyperbolic. 
By Witt Cancellation \cite[(8.4)]{EKM} we obtain that $\psi_K$ is hyperbolic.

Let $\mathsf{C}(\psi)$ denote the Clifford algebra of $\psi$.
Since $\psi$ is nondegenerate,  $\mathsf{C}(\psi)$ is a central simple
$F$-algebra with $\dim_F \mathsf{C}(\psi)=2^{\dim(\psi)}=2^6$. 
Moreover, the definition of $\psi$ as the squaring map on an
$F$-subspace of $A$ gives rise to an $F$-algebra homomorphism
$\mathsf{C}(\psi)\to A$, which is injective, because
$\mathsf{C}(\psi)$ is simple. 
Note that $\deg A=8$ as $(A,\s)$ is symplectic of capacity $4$.
Hence $\dim_F\mathsf{C}(\psi)=\dim_F A<\infty$. We 
conclude that $A\simeq \mathsf{C}(\psi)$.

In particular, as $\psi_K$ is hyperbolic, $A_K$ is split.
Therefore $\ind A\leq [K:F]=2$.
 Hence $A$ is Brauer equivalent to an $F$-quaternion algebra $Q$.
Since $(A,\s)$ is totally decomposable, 
 it follows by \cite[Theorem 6.1]{Dol17} that $(A,\s)\simeq
 \Ad(\varphi)\otimes (Q,\can_Q)$ for some bilinear $2$-fold Pfister
 form $\varphi$ over $F$. 

Let $\rho$ be the quadratic $4$-fold Pfister form over $F$ given by
the tensor product of  
$\varphi$ with the norm form of $Q$.
If $\rho$ is isotropic, then since it is a quadratic Pfister form, it
is hyperbolic, and it follows by \cite[Theorem 5.2]{Dol17} that
$(A,\s)$ is hyperbolic. 
Denoting by $F(\rho)$ the function field of the projective quadric
defined by $\rho$ over $F$, we obtain in any case that
$(A,\s)_{F(\rho)}$ is hyperbolic. 
On the other hand, since quadratic Pfister forms are either
anisotropic or hyperbolic, it follows by the Subform Theorem
\cite[(22.5)]{EKM} that every anisotropic $3$-fold Pfister form over
$F$ remains anisotropic over $F(\rho)$. 
Hence we may take $F'=F(\rho)$.
\end{proof}

\section{Construction of the discriminant from a biquadratic algebra\label{sec:biquad}} 

In this section we study an $F$-algebra with involution $(A,\s)$ with 
\begin{equation*}
1\in\Symd(\s)\quad\mbox{ and }\quad \kap(A,\s)=4\,.\end{equation*}
The first hypothesis excludes the case where $\s$ is orthogonal and  $\car(F)= 2$.
The second hypothesis means that 
\begin{equation*}
\dim_FA=2^{n+3} \quad\mbox{ for }\quad n=\begin{cases} 1 & \textrm{if $\s$ is orthogonal,}\\ 2 & \textrm{if $\s$ is unitary,}\\ 3 & \textrm{if $\s$ is symplectic.}\end{cases}\end{equation*}
We will attach to $(A,\s)$ a quadratic form over $F$ of dimension $2^n$ and study its properties. This form will yield a criterion for the decomposability of $(A,\s)$.
The construction of a candidate for this form crucially relies on the
existence of a biquadratic neat  $F$-subalgebra $L$ of $(A,\s)$, which
was proven in \cite[Theorem~7.4]{BGBT18a}.
By~\cite[Theorem~4.1]{BGBT18a}, the condition that $1\in\Symd(\s)$ implies that every neat
  subalgebra of $(A,\s)$ is contained in $\Symd(\s)$. It also implies that
  $x^2\in\Symd(\s)$ for all $x\in\Symm(\s)$, for if $\ell\in A$ is such
  that $\ell+\s(\ell)=1$, then $x^2 = x\ell x + \s(x\ell x)$.

\smallskip
By a \emph{biquadratic $F$-algebra} we mean a commutative $F$-algebra of dimension $4$ which is isomorphic to a tensor product of two quadratic $F$-algebras.
Let $L$ be an \'etale biquadratic $F$-algebra, and for an $F$-automorphism $\gamma$, let
$$L^\gamma=\{x\in L\mid \gamma(x)=x\}\,.$$
There are precisely three $F$-automorphisms $\gamma$ of order two for which $L^\gamma$ is a quadratic \'etale extension of $F$.
Together with $\id_L$, they form a subgroup $G$ of the automorphism group of $L$ such that the fixed field $L^G$ is equal to $F$. (In other terms, $L$ is a {\em $G$-Galois algebra} or a {\em Galois $F$-algebra} if $G$ is clear from the context, see \cite[(18.15)]{BOI}.)
We refer to $G$ as the \emph{Galois group of $L$ over $F$}.

\begin{lem}\label{L:quaternionizer}
Assume that $\kap(A,\s)=4$.
Let $L$ be a biquadratic neat  $F$-subalgebra of $(A,\s)$ and let
$\gamma$ be a nontrivial element of the Galois group of $L/F$. Let
$s:L^\gamma\to F$ be an $F$-linear form with $\ker(s)=F$ and
let $$W=\{x\in \Symd(\s)\mid yx=x\gamma(y)\mbox{ for all }y\in
L\}\,.$$ 
Then $\dim_FW=2^n$ for $n=\log_2\dim_FA-3$.
Moreover, $x^2\in L^\gamma$ for all $x\in W$, and
\[
  q\colon W\to F,\qquad x\mapsto s(x^2)
\]
is a nondegenerate quadratic form over $F$.
This form is isotropic if and only if there exists an element $x\in W$
with $x^2\in\mg{F}$.  
\end{lem}

\begin{proof}
We set $K=L^\gamma$, $C=\C_A(K)$, $\s_C=\s|_C$ and
$V=C\cap\Symd(\s)$. Note that $W$ is an $F$-subspace of $V$.
By \Cref{prop:maximal}, we have $\C_A(L)=L$.
For $x\in W$, it follows that $x^2\in\C_A(L)=L$ whereas $x\in C\setminus L$, whereby $x^2\in K[x]\cap L=K$.
Therefore $x\mapsto s(x^2)$ defines a quadratic form $q:W\to F$.
In order to prove that
$q$ is nondegenerate, we relate it to a nondegenerate quadratic form
$\wt{q}$ on $V$. We will see that $\wt{q}$ restricts to $q$ on
$W$ and is Witt equivalent to $q$. For the definition of $\wt{q}$, we consider
separately the cases where $K$ is a field and where $K$ is split.
 
Assume first that $K$ is a field.
Then $(C,\s_C)$ is a $K$-algebra with involution, which is of the same type as $(C,\s)$, in view of \cite[Proposition 4.12]{BOI}.
Therefore $\kap(C,\s_C)=2$ and $L$ is a  maximal neat
subalgebra of $(C,\s_C)$. 
We obtain by \cite[Proposition~4.6, Proposition~6.1 and
Proposition~6.5]{BGBT18a} that
\[
  V=\Symd(\s_C)=L\oplus W, \qquad\dim_FW={\textstyle\frac{1}8}\dim_FA,
\]
and that
there exists a canonical nondegenerate quadratic form
$c_2\colon V\to K$ with the following properties:
$W$ is the orthogonal complement of $L$, the restriction $c_2|_L$
is the norm form $\mathsf{N}_{L/K}$ of the quadratic \'etale extension $L/K$,
and $c_2(x)=-x^2$ holds for all $x\in W$.
It follows that the transfer $\wt{q}=-(s\circ c_2)\colon V\to
F$ is a 
nondegenerate quadratic form over $F$ such that $\wt{q}(x)=s(x^2)$ for
all $x\in W$, whereby $q$ is the restriction of $\wt{q}$ to
  $W$. Moreover, $W$ and $L$ are orthogonal to one another with
  respect to $\wt{q}$, and since $\dim_FV=\dim_FW+\dim_FL$, it follows
  that $W$ is  the orthogonal complement of $L$ in $V$ 
with respect to $\wt{q}$.
This implies that $q$ is nondegenerate and
\[
  \wt{q}= -(s\circ \mathsf{N}_{L/K}) \perp q.
\]
Since $L$ is a biquadratic $F$-algebra, the form $\mathsf{N}_{L/K}$ is
extended from a quadratic form defined over $F$. 
Hence $s\circ
\mathsf{N}_{L/K}$ is hyperbolic and $\wt{q}$ is Witt equivalent to $q$.

Since $s(1)=0$, it is clear that $q$ is isotropic whenever there
  exists 
$x\in W$ such that $x^2\in F^\times$. Conversely, assume that $q$ is isotropic.
Hence there exists $y\in W\setminus\{0\}$ such that $q(y)=0$, whereby
$y^2\in \ker(s)=F$. 
If $y^2\neq0$, then let $x=y$, and otherwise, the regular quadratic form
$c_2\rvert_W\colon W\to K$, $x\mapsto -x^2$ is isotropic and hence
universal, so we can find an element $x\in W$ with $x^2=1$.
This shows that $q$ is isotropic if and only if there exists $x\in W$
with $x^2\in\mg{F}$. 

Consider now the case where $K\simeq F\times F$.  We denote by $e_1$
and $e_2$ the two primitive idempotents of $K$.  Since $L$ is a
biquadratic \'etale $F$-algebra, there exists an $F$-automorphism
$\gamma'$ of $L$ which is of order $2$ and different from
$\gamma$. Then $\gamma'|_K\neq \id_K$, and it follows that $\gamma'$
interchanges $e_1$ and $e_2$.  Since $L=e_1L\oplus e_2L$, we conclude
by \cite[Lemma~5.8]{BGBT18a} that $K$ is neat in $(A,\s)$.  For $i=1$,
$2$ we set $E_i=e_iAe_i$ and denote by $\tau_i$ the restriction of
$\sigma$ to $E_i$.  Then $C=E_1\oplus E_2$ and the $F$-algebras $E_1$
and $E_2$ are isomorphic. Moreover, $(E_1,\tau_1)$ and $(E_2,\tau_2)$
are $F$-algebras with involution, and by~\cite[Proposition~3.1]{BGBT18a},
the involutions $\tau_1$ and $\tau_2$ have the same type as $\sigma$.
It follows that $\kap(E_1,\tau_1)=\kap(E_2,\tau_2)=2$.  For $i=1$,
$2$, since $e_iL$ is a quadratic \'etale $F$-subalgebra of $E_i$
contained in $\Symm(\tau_i)$ and with nontrivial automorphism
$\gamma|_{e_iL}$, we obtain by \cite[Proposition~4.6, Proposition~6.1 and
Proposition~6.5]{BGBT18a} that
\[
  e_iVe_i=e_iV=e_iL\oplus e_iW,\qquad
  \dim_Fe_iW={\textstyle\frac{1}{16}}\dim_FA,
\]
and that there exists a canonical nondegenerate 
quadratic form $c_{2,i}\colon e_iV\to e_iF$ with the following
properties: $e_iW$ is the orthogonal 
complement of $e_iL$, the restriction $c_{2,i}|_{e_iL}$ is the norm
form $\mathsf{N}_{e_iL/e_iF}$ of 
the quadratic \'etale  extension $e_iL/e_iF$, and $c_{2,i}(x)=-x^2$
holds for all $x\in e_iW$.
We consider the quadratic maps
$c_2\colon V\to K$, $x\mapsto c_{2,1}(e_1x)+c_{2,2}(e_2x)$ and
$\wt{q}=-(s\circ c_2)\colon V\to F$. The restriction of $c_2$ to
$L=e_1L\oplus e_2L$ is the norm form $\mathsf{N}_{L/K}$, and its restriction to
$W=e_1W\oplus e_2W$ is the map $W\to K,x\mapsto-x^2$.
The definition of $W$ implies that $W$ and $L$ are orthogonal to one
another with respect to $\wt{q}$. Since $\dim_FV=\dim_FW+\dim_FL$, we
obtain that 
\[
  \wt{q}=-(s\circ \mathsf{N}_{L/K})\perp q.
\]
Since $\ker(s)=F$, we have that $s(e_1)=-s(e_2)=\alpha$ for some $\alpha\in
F^\times$, whereby $\wt{q}\simeq-\alpha c_{2,1}\perp\alpha
c_{2,2}$. Therefore, the form $\wt{q}$ is nondegenerate, and the same
holds for $q$. Note further that  
$\dim_F W=\dim_Fe_1W+\dim_Fe_2W=\frac{1}8\dim_FA$ and that
the form $\mathsf{N}_{L/K}$ is extended from $F$.
We conclude that $s\circ \mathsf{N}_{L/K}$ is
hyperbolic and $q$ is Witt equivalent to $\wt{q}$.

Since $\ker(s)=F$, it is clear that $q$ is isotropic whenever there
exists $x\in W$ such that $x^2\in F^\times$. Conversely, if $q$ is
isotropic, we may find $w_1$, $w_2\in W$ with $(w_1,w_2)\neq (0,0)$
and $\lambda\in F$ such that 
$(e_iw_i)^2=e_i\lambda$ for $i=1,2$.
If $\lambda\neq0$, then for $x=e_1w_1+e_2w_2\in W$ we have
$x^2=\lambda\in F^\times$. If $\lambda=0$, then for $i=1,2$, the form
$c_{2,i}\rvert_{e_iW}$ is isotropic and
hence universal, so 
we may find $w_i'$ such that
$(e_iw_i')^2=e_i$, and  then $x=e_1w_1'+e_2w_2'\in W$ satisfies $x^2=1$.
\end{proof}

\begin{rem}
The construction of the form $\wt{q}$ in the proof of
\Cref{L:quaternionizer} goes back to \cite[Section~7]{GPT}, for the
case where $\car F\neq 2$ and $\s$ is symplectic. There it also turns
out, by a very different argument, that $\wt{q}$ contains a hyperbolic
plane and therefore is Witt equivalent to a smaller subform $q$. 
We will use in the sequel quite a different method than in \cite{GPT}
to determine the properties of $q$. 
\end{rem}

We define an operation $\ast:A\times A\lra A$ by letting
$$\hspace{2cm}x*y\,=\,xy+yx\quad\mbox{ for }\,\,x,y\in A\,.$$
(If $\car(F)\neq 2$, then $\ast$ corresponds up to a factor $2$ to the Jordan operation on~$A$.)
Note that for $x,y\in\Symm(\sigma)$ we have $\sigma(xy)=yx$ and thus
$x*y\in\Symd(\sigma)$.
              
\begin{thm}\label{L:mults}
Assume that $\kap(A,\s)=4$ and let $n=\log_2\dim_FA-3$.
Let $L$ be a biquadratic neat $F$-subalgebra of $(A,\s)$ and let $G=\{\id_L,\gamma_1,\gamma_2,\gamma_3\}$ be the Galois group of $L/F$.
For $1\leq i\leq 3$, we set $$W_i=\{x\in \Symd(\s)\mid yx=x\gamma_i(y)\mbox{ for all }y\in L\}\,.$$
There exists a quadratic $n$-fold Pfister form $\pi$ over $F$ 
and, for $1\leq i\leq 3$, an $F$-linear form $s_i:L^{\gamma_i}\to F$ with $\ker(s_i)=F$
such that $q_i:W_i\lra F,x\mapsto s_i(x^2)$ is a quadratic form over $F$ similar to $\pi$.
The Pfister form $\pi$ is uniquely determined by $L$.
Furthermore, we have
$$\Symd(\s)=L\oplus W_1\oplus W_2\oplus W_3\,.$$
\end{thm}

\begin{proof}
Let $i\in\{1,2,3\}$.
We set
$L_i=L^{\gamma_i}$, which is a quadratic neat $F$-subalgebra of $(A,\s)$.
We denote by $\Tr_i\colon L_i\to F$ its trace map and set $\Delta_i=\ker(\Tr_i)$.
Note that $\Delta_i$ is a $1$-dimensional $F$-subspace of $L_i$, which coincides with $F$ when $\car(F)=2$. 
The trace map $F\times F\to F$ is simply given by the addition in $F$.
Since $L_i$ is either a field or isomorphic to $F\times F$, we see that $\Delta_i\setminus\{0\}\subseteq \mg{L_i}$.
By \Cref{L:quaternionizer}, we have $\dim_F W_i=2^n$ and,  for every $x\in W_i$, we have $x^2\in L_i$ and consequently $\Tr_i(x^2)-2x^2\in\Delta_i$. This defines a map $$\vf_i:\ W_i\to \Delta_i,\ x\mapsto \Tr_i(x^2)-2x^2\,.$$
We will use various times in the sequel that, for distinct $i,j\in\{1,2,3\}$, $\gamma_j$ restricts to the nontrivial $F$-automorphism on $L_i$ and hence $\Tr_i=\id_{L_i}+\gamma_{j}|_{L_i}$.

A major step in the proof will now consist in showing that, for all $x\in W_1$ and $y\in W_2$, we have $xy,yx,x\ast y\in W_3$ and $$\vf_3(x\ast y) = \vf_1(x)\cdot \vf_2(y)\,.$$

Let $x\in W_1$ and $y\in W_2$.
Since $\gamma_3=\gamma_2\circ\gamma_1$, for $z\in L$ arbitrary, we have 
$$z(xy)=zxy=x\gamma_1(z)y=
xy\gamma_2(\gamma_1(z))=xy\gamma_3(z).$$
This shows that $xy\in W_3$, and in the same way we obtain that $yx\in W_3$.
Therefore $(xy)^2,(yx)^2\in L_3$ and in particular $xyxy+yxyx\in L_3$.
Furthermore $x\ast y\in W_3$, and in particular $(x*y)^2\in L_3$.
Since $x\in W_1$, we have
  \[
  \gamma_1(xyxy+yxyx)\cdot x = x\cdot(xyxy+yxyx) = x^2yxy+xyxyx.
  \]
  Since $y\in W_2$, $yx\in W_3$, $\gamma_3\circ\gamma_2=\gamma_1$ and $x^2\in L_1$, we have
  \[
  yxyx^2=(yx)\gamma_2(x^2)y=\gamma_3(\gamma_2(x^2))yxy=x^2yxy,
  \]
  hence
  \[
  \gamma_1(xyxy+yxyx)\cdot x = yxyx^2+xyxyx = (xyxy+yxyx)\cdot x
  \]
  and therefore
  \[
  \bigl(\gamma_1(xyxy+yxyx)-(xyxy+yxyx)\bigr)\cdot x=0.
  \]
Now, in any \'etale quadratic algebra with nontrivial
  automorphism $\iota$, the elements of the form $\iota(z)-z$ are
  invertible if they are nonzero. Using this for $L_3$, the last equation yields that
  \[
  \gamma_1(xyxy+yxyx)=xyxy+yxyx.
  \]
This shows that 
$xyxy+yxyx\in L_1\cap L_3=F$ and in particular, $$\Tr_3(xyxy+yxyx)=2(xyxy+yxyx)\,.$$
Since
  \begin{eqnarray*}
    (x*y)^2 & =&\,(xyxy+yxyx) + xy^2x + yx^2y\\
    & = &\,(xyxy+yxyx) + x^2\gamma_1(y^2)+\gamma_2(x^2)y^2\, ,
  \end{eqnarray*}
 it follows that $x^2\gamma_1(y^2)+\gamma_2(x^2)y^2\in L_3$ and
\begin{eqnarray*}
\vf_3(x\ast y) & = & \Tr_3(x^2\gamma_1(y^2)+\gamma_2(x^2)y^2) - 2 (x^2\gamma_1(y^2)+\gamma_2(x^2)y^2)\\
               & = & x^2y^2 + \gamma_2(x^2)\gamma_1(y^2) -x^2 \gamma_1(y^2) -\gamma_2(x^2)y^2)\\
               & = & (\gamma_2(x^2) - x^2)(\gamma_1(x^2) - y^2)\,.
\end{eqnarray*}
Since $\vf_1(x)=\Tr_1(x^2)-2x^2 = \gamma_2(x^2)-x^2$ and similarly $\vf_2(y)=\gamma_1(y^2)-y^2$, 
we thus have shown that 
\begin{eqnarray*}
\vf_3(x\ast y) & = & \vf_1(x) \cdot \vf_2(y)\,.
\end{eqnarray*}

We will now use this formula to show that, for a good choice of the linear forms $s_i:L_i\to F$,
we obtain a composition formula for the induced quadratic forms $q_i:W_i\to F,x\mapsto s_i(x^2)$, where $1\leq i\leq 3$.

For $i=1$, $2$, we fix some $\delta_i\in\Delta_i\setminus\{0\}$, and we set $\delta_3=\delta_1\delta_2$.
Since $\delta_1,\delta_2\in\mg{L}$ we have $\delta_3\in\mg{L}$, and since $\gamma_2\circ \gamma_1=\gamma_3=\gamma_1\circ\gamma_2$, we have $\delta_3 \in L_3$.
Since $\gamma_1$ restricts to the nontrivial $F$-automorphism on $L_i$ for $i=2,3$, we obtain that 
$$\Tr_3(\delta_3) =\delta_1\delta_2 + \gamma_1(\delta_1\delta_2)=\delta_1\Tr_2(\delta_2)=0\,.$$
Hence $\delta_3\in\Delta_3\setminus\{0\}$.  

Let $i\in\{1,2,3\}$. Since $\dim_F\Delta_i=1$, we have $\Delta_i = F\delta_i$.
We may now define 
the $F$-linear form $s_i: L_i\to F,x\mapsto (\Tr_i(x)-2x)\delta_i^{-1}$ to obtain that $\ker(s_i)=F$, because for $x\in L_i$ we have
  $\Tr_i(x)=2x$ if and only if $x\in F$. 
By \Cref{L:quaternionizer}, $q_i:W_i\to F, x\mapsto s_i(x^2)$ is a nongenerate quadratic form over $F$, and by  the choice of $s_i$, we have
  \[
    q_i(x)=(\Tr_i(x^2)-2x^2)\delta_i^{-1}=\varphi_i(x)\delta_i^{-1}
    \qquad\text{for $x\in W_i$.}
  \]
 Now, for $x\in W_1$ and $y\in W_2$, since 
$\varphi_3(x*y)=\varphi_1(x)\cdot\varphi_2(y)$ holds, we find that
$$
    q_3(x*y) = \varphi_3(x*y)\delta_3^{-1} =
    \varphi_1(x)\delta_1^{-1} \cdot \varphi_2(y)\delta_2^{-1}
    =q_1(x)\cdot q_2(y).
$$
We have thus shown that 
$$\ast\colon (W_1,q_1)\times (W_2,q_2)\lra (W_3,q_3)$$ 
is a composition of nondegenerate quadratic forms
in the sense of~\cite[p.~488]{BOI}.
As these forms are of dimension $2^n$,
it follows by \cite[(33.18) and (33.27)]{BOI} that
they are similar to one and the same quadratic $n$-fold Pfister form
$\pi$ over $F$. 

Since $\pi$ is a Pfister form, it is uniquely determined up to
isometry by its similarity class, which is also given by each of the
forms $q_1$, $q_2$ and $q_3$.  Since the linear forms
$s_i\colon L^{\gamma_i}\to F$ for $i=1$, $2$, $3$ are determined up to
a scalar by the property that $\ker(s_i)=F$, it follows that $\pi$
does also not depend on the choices of $s_1$, $s_2$, $s_3$. Therefore
$\pi$ is determined up to isometry by $L$.

We have $\dim_F\Symd(\s)=4+3\cdot 2^n=\dim_FL+\dim_FW_1+\dim_FW_2+\dim_FW_3$ and $L+W_1+W_2+W_3\subseteq\Symd(\s)$.
Hence, to prove finally that
$$\Symd(\s)=L\oplus W_1\oplus W_2\oplus W_3\,,$$ we only need to show that the sum 
$L+W_1+W_2+W_3$ is direct.

We fix an element $a\in L_1\setminus F$ with $\Tr_1(a)=1$. 
Then $2a-1\in \Delta_1\setminus\{0\}$ and hence
$(2a-1)^2\in\mg{F}$.
We further have $\gamma_2(a)=\gamma_3(a)=1-a$.
Hence for any $x\in L+W_1$ we have $ax+xa=2ax$, whereas for any $x\in W_2+W_3$ we obtain that  $ax+xa=x$.
Since $2a-1\in\mg{L}$ it follows that $(L+W_1)\cap (W_2+W_3)=0$.
In the same way we obtain that $(L+W_2)\cap (W_1+W_3)=0$.
Together this implies that the sum $L+W_1+W_2+W_3$ is direct.
 \end{proof}

\begin{rem}
If $\car(F)=2$, then in the proof of \Cref{L:mults},  for $1\leq i\leq 3$, we have $\ker(\Tr_i)=F$ and we may thus take $\delta_i=1$, whereby $s_i=\Tr_i$.
\end{rem}

For a biquadratic neat  $F$-subalgebra $L$ of $(A,\s)$,  we denote by
$\Pf_{\s,L}$ the quadratic Pfister form over $F$ which is
characterised in \Cref{L:mults}.
We will see later that this Pfister form does actually not depend on
the choice of $L$. It is clearly functorial:

\begin{prop}\label{P:biquat-Pfi-functorial}
Assume that $\kap(A,\s)=4$.
Let $L$ be a biquadratic neat  $F$-subalgebra of $(A,\s)$ and let
$F'/F$ be a field extension. 
Then $L\otimes F'$ is a biquadratic neat  $F'$-subalgebra of $(A,\s)_{F'}$
and 
the associated Pfister form $\Pf_{\s_{F'},L\otimes F'}$ is obtained by
scalar extension to $F'$ from $\Pf_{\s,L}$.
\end{prop}
Recall in this context that all tensor products are taken over $F$.
\smallskip

\begin{proof}
See \cite[Proposition~5.5]{BGBT18a} for the fact that $L\otimes F'$
is neat in $(A,\s)_{F'}$. 
The remaining part follows directly from the definitions and from
\Cref{L:mults}. 
\end{proof}

We next give a criterion for the hyperbolicity of $\Pf_{\s,L}$.

\begin{prop}\label{P:Pf-dec-equiv}
Assume that $\kap(A,\s)=4$.  
Let $L$ be a biquadratic neat  $F$-subalgebra of $(A,\s)$.
Let $K$ be a quadratic neat  $F$-subalgebra of $(A,\s)$ contained in $L$. The following conditions are equivalent:
\begin{enumerate}[$(i)$]
\item
  $\Pf_{\s,L}$ is hyperbolic.
\item
  $(A,\s)$ is decomposable along $L$.
\item
  $(A,\s)$ is decomposable along $K$.
\end{enumerate}
\end{prop}
\begin{proof}
Let $G$ be the Galois group of $L/F$.
Since $K$ and $L$ are neat in $(A,\s)$, it follows that 
$L$ is free as a $K$-module and hence $K$ is the fixed subalgebra of some element of $G$.
We write $G=\{\id_L,\gamma_1,\gamma_2,\gamma_3\}$ with $K=L^{\gamma_1}$.
For $i=1$, $2$, $3$ we set
$W_i=\{x\in \Symd(\s)\mid yx=x\gamma_i(y)\mbox{ for all }y\in L\}$ and
fix an $F$-linear form $s_i\colon L^{\gamma_i}\to F$ with
$\ker(s_i)=F$, and we consider the quadratic form $q_i\colon W_i\to
F$, $x\mapsto s_i(x^2)$. 
Then $q_1$, $q_2$ and $q_3$ are similar to the Pfister form
$\Pf_{\s,L}$. 
Hence $\Pf_{\s,L}$ is hyperbolic if and only if any of the forms
$q_1$, $q_2$ or $q_3$ is isotropic, and in this case they are all
isotropic. 

$(i\Rightarrow ii)$: Suppose that $\Pf_{\s,L}$ is hyperbolic. Then $q_1$ is isotropic, and by \Cref{L:quaternionizer}, there exists an element $x\in W_1$ with
$x^2\in\mg{F}$. 
Then $Q_2=L^{\gamma_2}\oplus L^{\gamma_2}x$ is a $\s$-stable $F$-quaternion algebra contained in $(A,\s)$. 
Since $\s$ is the identity on $L^{\gamma_2}$, the involution
$\s\rvert_{Q_2}$ is orthogonal. 
Let $A'=\C_A(Q_2)$ and $\s'=\s|_{A'}$.
It follows that $(A',\s')$ is an $F$-algebra with involution with
$\kap(A',\s')=2$. 
Since $K=L^{\gamma_1}\subseteq A'$, it follows by \cite[Corollary 6.6]{BOI}
that $K$ is contained in a $\s$-stable $F$-quaternion subalgebra $Q_1$ of $A'$, and hence of $A$. 
Hence $Q_1$ and $Q_2$ are independent $\s$-stable $F$-quaternion subalgebras of $A$ such that  
$L=(L\cap Q_1)\cdot(L\cap Q_2)$. Therefore $(A,\s)$ is decomposable
along $L$.

$(ii\Rightarrow iii)$: This implication is obvious.

$(iii\Rightarrow i)$: 
Recall the quadratic form $\wt{q}\colon \C_A(K)\cap\Symd(\s)\to F$
defined in the 
proof of \Cref{L:quaternionizer} as a transfer with respect to an $F$-linear form
$s\colon K\to F$. This form $\wt{q}$ is Witt equivalent to
$q_1$ and hence similar to $\Pf_{\s,L}$.
Hence to prove $(i)$ it suffices to show that $\wt{q}$
is hyperbolic. Condition~$(iii)$ yields a
decomposition $A=Q\otimes B$ with a $\sigma$-stable $F$-subalgebra $B$ of $A$ and a $\sigma$-stable $F$-quaternion subalgebra of $A$ containing $K$.
For $C=\C_A(K)$ we obtain that
\[
  C=K\otimes B.
\]
Let $\s_B=\s\rvert_B$. Since $\s$ is the identity on $K$, the
involution $\s\rvert_Q$ is orthogonal.
Hence $\s_B$ has the same type as $\s$ and $\kap(B,\s_B)=2$. 
The quadratic form $c_2$ on
$\Symd(\s_C)$  over $K$ given by \cite[Proposition~4.6]{BGBT18a} is extended from the corresponding quadratic form $c_2$
on $\Symd(\s_B)$ over $F$, hence its transfer with respect to $s\colon K\to F$ is hyperbolic. 
Since $\wt{q}=-s\circ c_2$, it follows that $(i)$ holds.
\end{proof}

\begin{cor}\label{C:cap4-hyp-decalong}
Assume that $\kap(A,\s)=4$.
If $\s$ is hyperbolic and $\exp A\leq 2$,
then $\Pf_{\s,L}$ is hyperbolic for every biquadratic neat  $F$-subalgebra
$L$ of $(A,\s)$. 
\end{cor}
\begin{proof}
Let $L$ be a biquadratic neat  $F$-subalgebra of $(A,\s)$.
We fix a quadratic neat  $F$-subalgebra $K$ of $(A,\s)$ contained in $L$. If $\s$ is hyperbolic, then \Cref{T:hyp-quad-factor}
shows that $(A,\s)$ is decomposable along $K$, and hence $\Pf_{\s,L}$ is
hyperbolic by \Cref{P:Pf-dec-equiv}.
\end{proof}

\begin{thm}\label{C:dec-along-all-quad-neat}
Assume that $(A,\s)$ is totally decomposable and $\kap(A,\s)=4$.
Then $(A,\s)$ is decomposable along every biquadratic neat  $F$-subalgebra of $(A,\s)$.
\end{thm}
\begin{proof}
Let $L$ be a biquadratic neat  $F$-subalgebra of $(A,\s)$.
By \Cref{L:make-decomp-hyp} there exists a field extension $F'/F$ such that 
$(A,\s)_{F'}$ is hyperbolic and such that every anisotropic $n$-fold Pfister form over $F$ remains anisotropic over $F'$.
Note that $\exp A_{F'}\leq \exp A\leq 2$, because $(A,\s)$ is totally decomposable.
Since $\s_{F'}$ is hyperbolic, it follows by \Cref{P:biquat-Pfi-functorial} and  \Cref{C:cap4-hyp-decalong} that $(\Pf_{\s,L})_{F'}$ is hyperbolic.
By the choice of $F'/F$, we conclude that $\Pf_{\s,L}$ is hyperbolic.
Therefore $(A,\s)$ is decomposable along~$L$, by \Cref{P:Pf-dec-equiv}.
\end{proof}

\begin{cor}\label{C:kap4ttdec-splitfactor}
Assume that
$(A,\s)$ is totally decomposable, $\kap(A,\s)=4$ and  $\coind A$ is even.
Then $(A,\s)$ has a total decomposition involving a split
$F$-quaternion algebra with orthogonal involution. 
\end{cor}

\begin{proof}
  If $\s$ is hyperbolic then the statement follows directly from
  \Cref{L:meta-quat-factor}.  Hence we may assume that $\s$ is not
  hyperbolic.  Since $\coind(A,\s)$ is even, it follows from
  \cite[Corollary~5.12]{BGBT18a} that $(A,\s)$ contains a split
  quadratic neat $F$-subalgebra $K$.  By \cite[Theorem~6.10]{BGBT18a},
there is a quadratic neat  $F$-subalgebra $K'$ of $(A,\s)$ such
    that $K\otimes K'$ is a biquadratic neat  $F$-subalgebra of $(A,\s)$.
    We set $L=K\otimes K'$. By \Cref{C:dec-along-all-quad-neat}, 
    $(A,\s)$ is decomposable along $L$.
    It follows by \Cref{P:arrange-decompalong} that there exist
    $\s$-stable $F$-quaternion subalgebras $Q$ and $Q'$ of $A$ such that
    $Q\cap L=K$ and $Q'\cap L=K'$.
    As $K$ is
  split, the quaternion algebra $Q$ is split.  Since $\s$ is not
  hyperbolic, we conclude that $\s|_Q$ is not hyperbolic.  Hence
  $\s|_Q$ is orthogonal, for the unique symplectic involution on a
  split quaternion algebra is hyperbolic.
\end{proof}

We are now in the position to show that the form $\Pf_{\s,L}$ given by
a biquadratic neat  subalgebra $L$ of $(A,\s)$ is independent of the
choice of this subalgebra. 
\begin{prop}\label{C:Pf-disc-def}
Assume that $\kap(A,\s)=4$ and let $n=\log_2\dim_FA-3$.
There exists a unique $n$-fold Pfister form $\pi$ over $F$ such that
$\pi\simeq \Pf_{\s,L}$ for every biquadratic neat  $F$-subalgebra $L$
of $(A,\s)$. 
\end{prop}
\begin{proof}
By \cite[Theorem~7.4]{BGBT18a} there exists a biquadratic neat 
$F$-subalgebra $L$ of $(A,\s)$. 
Consider another biquadratic neat  $F$-subalgebra $L'$ of $(A,\s)$.
We need to show that $\Pf_{\s,L'}=\Pf_{\s,L}$.

If $(A,\s)$ is totally decomposable, then we obtain by
\Cref{C:dec-along-all-quad-neat} and \Cref{P:Pf-dec-equiv}  that
$\Pf_{\s,L}$ and $\Pf_{\s,L'}$  
are hyperbolic, whereby $\Pf_{\s,L'}= \Pf_{\s,L}$.
Assume now that $(A,\s)$ is not totally decomposable.
Then it follows by \Cref{P:Pf-dec-equiv} that $\Pf_{\s,L}$ and
$\Pf_{\s,L'}$ are both anisotropic. 
Let $F'$ denote the function field of the projective quadric over $F$
given by $\Pf_{\s,L}$. 
Then $\Pf_{\s,L}$ becomes hyperbolic over $F'$.
By \Cref{P:biquat-Pfi-functorial} and \Cref{P:Pf-dec-equiv}, it follows that 
$(A,\s)_{F'}$ is decomposable along $L\otimes F'$.
By \Cref{C:dec-along-all-quad-neat}, then $(A,\s)_{F'}$ is also
decomposable along $L'\otimes F'$.  
By \Cref{P:biquat-Pfi-functorial} and \Cref{P:Pf-dec-equiv}, it
follows that $\Pf_{\s,L'}$ becomes hyperbolic over $F'$. 
Since $\Pf_{\s,L}$ and $\Pf_{\s,L'}$ are anisotropic $n$-fold Pfister
forms and since $F'$ is the function field of $\Pf_{\s,L}$ over $F$,
we conclude by the Subform Theorem \cite[(22.5)]{EKM} that
$\Pf_{\s,L'}=\Pf_{\s,L}$. 
\end{proof}

When $\kap(A,\s)=4$, we denote by $\Pf_\s$ the quadratic Pfister form
over $F$ which is characterised in \Cref{C:Pf-disc-def} and we call it
the \emph{discriminant Pfister form of $(A,\s)$}. 

\begin{thm}\label{T:Pf-disc-hyp}
Assume that $\kap(A,\s)=4$. 
The form $\Pf_{\s}$ is hyperbolic if and only if $(A,\s)$ is totally decomposable.
Moreover, for any field extension $F'/F$, we have that $(\Pf_{\s})_{F'}=\Pf_{\s_{F'}}$.
\end{thm}
\begin{proof}
By \cite[Theorem~7.4]{BGBT18a}, there exists a biquadratic neat 
$F$-subalgebra of $(A,\s)$. 
Hence the first part follows from \Cref{P:Pf-dec-equiv} and
\Cref{C:dec-along-all-quad-neat}. The second part follows from
\Cref{P:biquat-Pfi-functorial}. 
\end{proof}

This finishes the proof of our main theorem stated in the introduction.

\section{Determination of the discriminant Pfister form} 
\label{S:final}

In this final section we relate the discriminant Pfister form of an
algebra with involution of capacity $4$ to other known invariants and
we compute it in some special situations.
We show in~\Cref{P:cap4unitary-invarel} that in the unitary
  case the discriminant Pfister form is the norm form of a quaternion
  algebra that is Brauer equivalent to the discriminant algebra, and
  in~\Cref{cor:KQ} we use this result to give an alternative proof of
  the criterion due to Karpenko--Qu\'eguiner \cite{KQ} for total decomposability of
  algebras with unitary involution of degree~$4$. In the symplectic
  case when the characteristic is
  different from~$2$, we show in~\Cref{P:GPT} that the cohomological invariant
 introduced by Garibaldi--Parimala--Tignol in \cite{GPT} is the Arason invariant of the discriminant
  Pfister form, and we give an alternative proof of the result from \cite{GPT}
  characterising totally decomposable involutions by the vanishing of
  this invariant.

We further aim to relate the notions of discriminant Pfister form for
the different types of algebras with involution of capacity $4$.  
We will see
in~\Cref{L:discneatquadnorm}
  that an embedding between two such algebras with involution leads
  naturally to a factorisation relation between the corresponding
  Pfister forms. 

  In the sequel, we shall use the notion of quaternion $K$-algebra
  also in cases where $K$ is an \'etale $F$-algebra, but not
  necessarily a field. More generally, quaternion $K$-algebras can be
  defined over an arbitrary commutative ring $K$: see
  \cite[p. 4]{Knus} for the general case or \cite[p. 27]{Bae} in the
  case where $K$ is a semilocal ring (which therefore covers the case
  where $K$ is an \'etale $F$-algebra).  The Skolem-Noether theorem,
  familiar in the case where $K$ is a field, extends to the case where
  $K$ is an \'etale $F$-algebra: this follows directly by applying it
  from the situation where the center is a field to the simple
  components of $K$ (whose centers are the simple factors of $K$,
  hence fields).
 
\begin{lem}\label{L:cap4unitarybis}
Let $(A,\s)$ be an $F$-algebra with unitary involution of capacity
$4$.  Let $L$ be a biquadratic neat  $F$-subalgebra of $(A,\sigma)$
with Galois group $\{\id_L,\gamma_1,\gamma_2,\gamma_3\}$ and let
$Z=\Z(A)$. 
The following hold:
\begin{enumerate}[$(1)$]
\item  There exists $w\in \Symm(\sigma)\cap A^\times$ such that 
$w\ell = \gamma_1(\ell)w$ for all $\ell\in L$.
\item For every $w$ as in $(1)$, we have $\Pf_\sigma \simeq
  \la1,-\Nrd_A(w)\ra\otimes\N_{N/F}$ for the quadratic \'etale
  $F$-algebra $N=(L^{\gamma_2}\otimes Z)^{\gamma_1\otimes\sigma\rvert_Z}$.
\end{enumerate}
\end{lem}

\begin{proof}
First observe that, after fixing an element $z\in Z\setminus F$ such that $z+\sigma(z)=1$, every element $a\in A$ can be decomposed as
$a=s_1+s_2z$ with $s_1, s_2\in\Symm(\sigma)$, given explicitly by $s_2=\bigl(z-\sigma(z)\bigr)^{-1}\bigl(a-\sigma(a)\bigr)$ and $s_1=a-s_2z$. 
Therefore multiplication in $A$ induces an isomorphism of $F$-vector spaces $$\Symm(\sigma)\otimes Z\to A\,.$$ 
This allows us to identify $L\otimes Z$ with the $F$-subalgebra $LZ$ of $A$. 
We obtain that $LZ$ is a Galois $F$-algebra: the nontrivial elements of its Galois group are $\sigma\rvert_{LZ}$ and the maps $\gamma_i\otimes\id_Z$ and
  $\gamma_i\otimes\sigma\rvert_Z$ for $i=1,2,3$.

To simplify notation we set $K=L^{\gamma_1}$.
Let $C=\C_A(K)=\C_A(KZ)$, which is a $KZ$-quaternion algebra and to which $\sigma$ restricts as a unitary
  involution. Let further
  \begin{equation}
    \label{eq:cap4unitarybis}
    D= \{x\in C\mid \sigma(x)=\can_C(x)\}.
  \end{equation}
Then $D$ is a $K$-quaternion algebra: if $KZ$ is a field, then this follows by~\cite[(2.22)]{BOI}, using that $K$ is the $F$-subalgebra of $KZ$ fixed under $\sigma\rvert_{KZ}$.
Clearly, $LZ\subseteq C=DZ$, and $\can_C$ restricts to $\gamma_1\otimes\id_Z$ on $LZ$ because $KZ$ is the center of $C$. Therefore the $F$-algebra $M=(LZ)^{\gamma_1\otimes\sigma\rvert_Z}$ lies in $D$, and it is a quadratic Galois extension of $K$ with Galois group generated by $\gamma_1\otimes\id_Z$. 
By the Skolem--Noether Theorem, $(\gamma_1\otimes\id_Z)\rvert_M$ extends to an inner automorphism $\Int_D(y)$ of $D$ for some $y\in D^\times$. 
Since $(\gamma_1\otimes\id_Z)^2\rvert_M=\id_M$, and $D=M[y]$, it follows that $y^2\in \C(D)=K$.
Since $(\gamma_1\otimes\id_Z)\rvert_M\neq \id_M$, we have $y\notin\C(D)=K$, and as $\s|_D=\can_D$ and $y^2\in K$, we conclude that $\s(y)=-y$. As $C=DZ$ and $MZ=LZ$, it follows that $\Int_C(y)|_{L}=\gamma_1$.

We pick an element $z'\in Z^\times$ such that $\sigma(z')=-z'$ and set $w_0=yz'$. (If $\car(F)=2$ then we may take $z'=1$.) 
Then $\sigma(w_0)=w_0$ and $w_0\in A^\times$, and for any $\ell\in L$ we have 
  \[
    w_0\ell w_0^{-1}=\Int_C(y)(\ell)= \gamma_1(\ell).
  \]
Hence $w_0$ is a possible choice for an element $w$ satisfying the conditions in~$(1)$.

To prove part $(2)$, we now consider an arbitrary element $w$ as in~$(1)$ while keeping our choice for $w_0$.
We set
\[
W=\{x\in\Symm(\sigma)\mid x\ell = \gamma_1(\ell)x\text{ for all $\ell\in L$}\},
\]
and fix an $F$-linear form $s\colon K\to F$ with $\ker(s)=F$. 
Since $\sigma$ is unitary, we have $\Symm(\sigma)=\Symd(\sigma)$, hence
\[
    q\colon W\to F,\qquad x\mapsto s(x^2)
\]
is the quadratic form from \Cref{L:quaternionizer}, and by definition the Pfister form $\Pf_\sigma$ is similar to $q$. Since $M=(LZ)^{\gamma_1\otimes\sigma\rvert_Z}$, we have $\sigma\rvert_M=(\gamma_1\otimes\id_Z)\rvert_M=\Int_A(w)\rvert_M$,
  hence
  \[
    \sigma(wm)=\sigma(m)w=wm \qquad\text{for all $m\in M$,}
  \]
which proves that $wm\in\Symm(\sigma)$. 
Since moreover $M$ centralizes $L$ it follows that $wM\subseteq W$. 
As $w\in\mg{A}$ we have $\dim_F wM=\dim_F M=4=\dim_F W$,
so we conclude that $$W=wM\,.$$
As $w_0\in W\cap\mg{A}$, we may
  write $w=w_0m_0$ for some $m_0\in M$, hence
\[
    w^2=w_0m_0w_0m_0=\sigma(m_0)w_0^2m_0= \sigma(m_0)y^2{z'}^2m_0.
\]
Since $y^2\in K$, ${z'}^2\in F$ and $\s(m_0)m_0=\N_{M/K}(m_0)\in K$, it follows that $w^2\in K$. For $m\in M$
  we have
  \[
    (wm)^2=w^2\sigma(m)m=w^2\N_{M/K}(m)\quad\in K\,.
  \]
  We conclude that $q$ is isometric to the quadratic form
  \[
    q'\colon M\to F,\qquad m\mapsto
    s\bigl(w^2\N_{M/K}(m)\bigr).
  \]
Note that $M=(LZ)^{\gamma_1\otimes\sigma\rvert_Z}$ is a Galois $F$-algebra whose Galois group is generated by $\sigma\rvert_M=\gamma_1\otimes\id_Z$ and $\gamma_2\otimes\id_Z$.
Since $\s|_K=\id_K$, it follows that $M=KN$ for $$N=M^{\gamma_2\otimes\id_Z}=(L^{\gamma_2}\otimes Z)^{\gamma_1\otimes\sigma\rvert_Z}$$ and $M$ is naturally isomorphic to $K\otimes N$.
Therefore the quadratic form $\N_{M/K}$ extends $\N_{N/F}$. 
By Frobenius Reciprocity \cite[(20.2), (20.3c)]{EKM} it follows that $q'$ is isometric to $s_\ast(\la w^2\ra)\otimes \N_{N/F}$. Set $d=\Nrd_A(w)$. 
Since $w\notin K$ and $w^2\in K$, we obtain that $d=\N_{K/F}(w^2)$. 
Since $s_\ast(\la 1\ra)$ is hyperbolic, it follows from \cite[(34.19)]{EKM} that $s_\ast(\la w^2\ra)$ is similar to $\la 1,-d\ra$. 
Therefore $q'$ is similar to $\la 1,-d\ra\otimes \N_{N/F}$. 
This shows that the forms $\Pf_\s$ and $\la 1,-d\ra\otimes \N_{N/F}$ are similar, and since they are both Pfister forms, we conclude that they are isometric.
\end{proof}

\begin{prop}\label{P:cap4unitary-invarel}
Let $(A,\s)$ be an $F$-algebra with unitary involution of capacity $4$. 
Then $\Pf_\s\simeq\Nrd_Q$ for an $F$-quaternion algebra $Q$, which is Brauer equivalent to the discriminant algebra of $(A,\s)$.
\end{prop}

\begin{proof}
We fix a biquadratic neat  $F$-subalgebra $L$ of $(A,\s)$, whose existence is guaranteed by \cite[Theorem~7.4]{BGBT18a}.
We use the same notation as in the proof of  \Cref{L:cap4unitarybis}.
By~\cite[p.~129]{BOI}, the $K$-quaternion algebra $D$ defined in
\eqref{eq:cap4unitarybis} is the discriminant algebra of
$(C,\s\rvert_C)$. By~\cite[Lemma~3.1(2)]{BFT}, the
discriminant  algebra of $(A,\sigma)$ is Brauer equivalent to the
corestriction of $D$ to $F$.  

The proof of \Cref{L:cap4unitarybis} yields that
\[D=M\oplus My = KN\oplus KNy\]
where $M$ is a biquadratic Galois $F$-algebra and $K$ and $N$ are the subalgebras fixed by two different nontrivial elements of the Galois group of $M/F$ and where $y\in\mg{D}$ is such that $\Int_D(y)$ extends the nontrivial $K$-automorphism of $M$, $\s(y)=-y$ and $y^2\in K$.
Hence $D$ is isomorphic to the crossed product algebra $(KN/K,y^2)$ over $K$.
Since $M=KN$, which is a free compositum over $F$, it follows by the projection formula (see~\cite[Prop.~3.4.10]{GS06}) that the corestriction of $D$ is Brauer equivalent to the crossed product algebra $Q=\bigl(N/F,\N_{K/F}(y^2)\bigr)$ over $F$, which is an $F$-quaternion algebra with norm form
  $\la 1,-\N_{K/F}(y^2)\ra\otimes\N_{N/F}$.

By the proof of \Cref{L:cap4unitarybis}, after choosing $z'\in\mg{\C(A)}$ with $\s(z')=-z'$ and letting $w_0=yz'$, we obtain that $\Pf_\sigma\simeq\la1,-\Nrd_A(w_0)\ra\otimes\N_{N/F}$.
Since ${z'}^2\in F^\times$ and $\Nrd_A(w_0)=\N_{K/F}(y^2){z'}^4$, it follows that $\Pf_\sigma$ is isometric to
  $\Nrd_Q$.
\end{proof}

We can now retrieve from \Cref{T:Pf-disc-hyp} the criterion from
\cite[Section~3]{KQ} for total decomposability of an algebra with
unitary involution of capacity $4$. 

\begin{cor}[Karpenko--Qu\'eguiner]
  \label{cor:KQ}
  An $F$-algebra with unitary involution of capacity $4$ is totally
  decomposable if and only if its discriminant algebra is split.
\end{cor}

\begin{proof}
  Let $(A,\s)$ be an $F$-algebra with unitary involution of capacity
  $4$ and let $D$ be its discriminant algebra.  By
  \Cref{P:cap4unitary-invarel}, $D$ is Brauer equivalent to an $F$-quaternion algebra $Q$ such that $\Pf_\s\simeq\Nrd_Q$.  It
  follows that $D$ is split if and only if $\Pf_\s$ is hyperbolic, and
  by \Cref{T:Pf-disc-hyp} this is equivalent to $(A,\s)$ being totally
  decomposable.
\end{proof}

When $\car F\neq 2$, for an $F$-algebra with symplectic
involution of degree a multiple of $8$, a cohomological invariant
$\Delta(A,\s)\in H^3(F,\mu_2)$ was defined in \cite{GPT}.
In the case where $\deg A=8$ this invariant is related to the discriminant Pfister form.

For a $3$-fold Pfister
form $\pi$ over a field $F$ of characteristic different from~$2$,
and for $a$, $b$, $c\in\mg{F}$ such that 
$\pi\simeq\la 1,-a\ra\otimes \la 1,-b\ra\otimes \la 1,-c\ra$, the
cup product $(a)\cup(b)\cup(c)$ in $H^3(F,\mu_2)$ is an
invariant of $\pi$, by~\cite[Satz 1.6]{Ara75}, also called the
\emph{Arason invariant of $\pi$}.  

\begin{prop}[Garibaldi--Parimala--Tignol]\label{P:GPT}
Assume that $\car F\neq 2$.  Let $(A,\s)$ be an $F$-algebra with
symplectic involution with $\kap(A,\s)=4$.  Then $\Delta(A,\s)$ is
the Arason invariant of $\Pf_\s$. Furthermore $\Delta(A,\s)=0$ if and only if $(A,\s)$ is totally decomposable.
\end{prop}

\begin{proof}
It is proven in \cite[Proposition~8.1]{GPT} that $\Delta(A,\s)$ is the
Arason invariant of a $10$-dimensional quadratic form of trivial
discriminant and Clifford invariant: this is the form $\wt{q}$
appearing the proof of \Cref{L:quaternionizer}. 
This form $\wt{q}$ is Witt equivalent to the $8$-dimensional quadratic
form $q$ in \Cref{L:quaternionizer}, and the Pfister form $\Pf_\s$ is
by definition similar to $q$, whereby its Arason invariant is the same
as for $q$ and $\wt{q}$. 
This relates the two invariants in the way as it is claimed here.

The equivalence of the vanishing of $\Delta(A,\s)$ with the
decomposability of $(A,\s)$ is shown in \cite[Section 9]{GPT}; it can
now alternatively be obtained from \Cref{T:Pf-disc-hyp}.  
In either way one relies on the fact that a quadratic $3$-fold Pfister
form is hyperbolic if and only if its Arason invariant is trivial,
which follows from \cite[Satz 5.6]{Ara75}.     
\end{proof}

We return to the situation where the field $F$ is of arbitrary
characteristic. 
For an involution $\s$ on an $F$-algebra $A$ one defines
$$
\Alt(\s)=\{x-\s(x)\mid x\in A\}\,.
$$
Recall from \cite[(7.2)]{BOI} that the discriminant of an orthogonal
involution $\s$ on a central simple $F$-algebra $A$ of even degree
$2m$ is the square 
class $(-1)^m\Nrd_A(y)\sq{F}$ in $\scg{F}$ given by an arbitrary
element $y\in \mg{A}\cap\Alt(\s)$, and that there always exists such
an element. 

\begin{lem}\label{L:discneatquadnorm}
Let $(B,\tau)$ be an $F$-algebra with orthogonal involution of degree
$4$. 
Let $d\in\mg{F}$ be such that the discriminant of $\tau$ is $d\sq{F}$.
Then $d$ is represented by $\N_{E/F}$ for every quadratic neat 
$F$-subalgebra $E$ of $(B,\tau)$.  
\end{lem}

\begin{proof}
Fix a quadratic neat  $F$-subalgebra $E$ of $(B,\tau)$.
Then $C=\C_B(E)$ is a quaternion $E$-algebra. We fix an element $y\in
\mg{C}\cap\Alt(\tau|_C)$.  
Since $\tau|_C$ is orthogonal, we have $y\notin E$.
On the other hand $y^2\in E$. Hence $\deg B=4=[E[y]:F]$, and it
follows that 
 $\Nrd_B(y)=\N_{E/F}(y^2)$. 
By the choice of $d$, we obtain that
$d\sq{F}=\Nrd_B(y)\sq{F}=\N_{E/F}(y^2)\sq{F}$, which shows the claim. 
\end{proof}

\begin{prop}
\label{prop:examples}
Let $(B,\tau)$ be an
  $F$-algebra with orthogonal involution of capacity $4$.  
  Let $d\in\mg{F}$ be such that $d\sq{F}$ is the discriminant of
  $(B,\tau)$. The following hold:
\begin{enumerate}[$(1)$]
\item If $\car F\neq 2$, then $\Pf_{\tau}=\la 1,-d\ra$.
\item For any quadratic \'etale $F$-algebra $Z$, the discriminant Pfister form of the $F$-algebra with unitary
    involution $(B,\tau)\otimes (Z,\can_{Z/F})$ of capacity $4$ is
    given by $\la 1,-d\ra\otimes \mathsf{N}_{Z/F}$. 
\item For any $F$-quaternion algebra $Q$, the discriminant
    Pfister form of the $F$-algebra with symplectic involution
    $(B,\tau)\otimes (Q,\can_Q)$ of capacity $4$ is given by $\la
    1,-d\ra\otimes \Nrd_Q$ where $\Nrd_Q$ is the reduced norm form of
    $Q$. 
\end{enumerate}
\end{prop}

\begin{proof}
By \cite[Theorem 7.4]{BGBT18a}, $(B,\tau)$ contains a biquadratic neat 
$F$-subalgebra $L$.  
Let $\{\id_L,\gamma_1,\gamma_2,\gamma_3\}$ be the Galois group of $L$
viewed as a Galois $F$-algebra.  
We set $K=L^{\gamma_1}$ and fix an $F$-linear functional $s\colon K\to
F$ with $\ker(s)=F$. 
We further set $C_0=\C_B(K)$ and observe that $C_0$ is a
$K$-quaternion algebra containing $L$ and that $\tau$ restricts to an
orthogonal involution on $C_0$. 
We fix $y\in C_0^\times\cap\Alt(\tau\rvert_{C_0})$. 
Then $y^2\in K$ and $y\in B^\times\cap\Alt(\tau)$, hence
  \begin{equation}
    \label{eq:yy}
    \can_{C_0}(y)=-y,\qquad \Nrd_B(y)=\N_{K/F}(y^2)
    \quad\text{and}\quad dF^{\times2}=\Nrd_B(y)F^{\times2}.
  \end{equation}
Moreover, since $y\in\Alt(\tau\rvert_{C_0})$ and $L\subseteq\Symm(\tau\rvert_{C_0})$, it follows that $\Trd_{C_0}(yL)=0$, by \cite[(2.3)]{BOI}. 
As $\can_{C_0}\rvert_L=\gamma_1\rvert_L$, this implies that
  \begin{equation}
    \label{eq:y}
    y\ell=\gamma_1(\ell)y\qquad\text{for every $\ell\in L$.}
  \end{equation}

$(1)$\, Recall that (only) for part $(1)$ we assume that $\car F\neq
2$. Thus, $\Symd(\tau)=\Symm(\tau)$ and $L=K\oplus vK$ for some
element $v\in L^\times$ with $v^2\in F^\times$ 
and $\gamma_1(v)=-v$. The latter implies that $yv=-vy$. Let
\[
  W_0=\{x\in\Symm(\tau)\mid x\ell = \gamma_1(\ell)x \text{ for all
    $\ell\in L$}\}.
\]
By definition, the Pfister form $\Pf_\tau$ is similar to the quadratic
form $$q_0\colon W_0\to F,\,x\mapsto s(x^2)\,.$$

As $\tau(y)=-y$, $\tau(v)=v$ and $yv=-vy$, it follows that
$yv\in\Symm(\tau)$, and as $v\in L$, we conclude by \eqref{eq:y} that $yv\in W_0$. 
Hence $yvK\subseteq W_0$, and since $yv\in B^\times$ and therefore $\dim_F\, yvK=\dim_FK=2=\dim_FW_0$, we obtain that
$$W_0=yvK\,.$$ For $x\in K$, we have
\[
  q_0(yvx)=s\bigl((yv)^2x^2\bigr) = -v^2s(y^2x^2).
\]
Therefore, $q_0$, hence also $\Pf_\tau$, is similar to $s_*(\la
y^2\ra)$.
Now, the discriminant of $s_*(\la y^2\ra)$ is
  $\N_{K/F}(y^2)F^{\times2}$ by~\cite[(34.19)]{EKM}, so
$s_*(\la y^2\ra)$ is similar to
$\la1,-\N_{K/F}(y^2)\ra$, hence also
to $\la1,-d\ra$ by~\eqref{eq:yy}. We conclude that $\Pf_\tau$ is
similar to $\la1,-d\ra$, and since both binary forms represent~$1$, it
follows that they are isometric.

$(2)$\,
For proving part $(2)$, we set $(A,\sigma)=
(B,\tau)\otimes(Z,\can_{Z/F})$ and fix an element $z\in Z^\times$ such that
$\can_{Z/F}(z)=-z$. (If $\car(F)=2$ we may choose $z=1$.) 
Then $yz\in\Symm(\sigma)\cap A^\times$ and \eqref{eq:y} shows that 
$yz\ell=\gamma_1(\ell)yz$ for every $\ell\in L$.
Hence it follows from
\Cref{L:cap4unitarybis} that
$\Pf_\sigma\simeq\la1,-\Nrd_A(yz)\ra\otimes\N_{N/F}$ for the quadratic \'etale $F$-algebra
$N=(L^{\gamma_2}\otimes Z)^{\gamma_1\otimes\can_{Z/F}}$. Since $z^2\in
F^\times$, we have by~\eqref{eq:yy} $\Nrd_A(yz)=\Nrd_A(y)z^4\in
dF^{\times2}$, whereby $\la1,-\Nrd_A(yz)\ra\simeq\la1,-d\ra$. To
complete the proof, it suffices to show that
$\la1,-d\ra\otimes\N_{N/F}$ is isometric to
$\la1,-d\ra\otimes\N_{Z/F}$.

For this we note that $(LZ)^{\gamma_2\otimes\id_Z}$ is a Galois $F$-algebra with Galois group isomorphic to $(\zz/2\zz)^2$, and the quadratic
$F$-subalgebras fixed under the nontrivial elements of the Galois
group are $N$, $Z$ and $L^{\gamma_2}$.
Hence
$N\otimes L^{\gamma_2}\simeq Z\otimes L^{\gamma_2}$. If
$L^{\gamma_2}$ splits, then $N\simeq Z$, hence
$\N_{N/F}\simeq\N_{Z/F}$ and the proof is complete. If $L^{\gamma_2}$
is a field, the isomorphism $N\otimes L^{\gamma_2}\simeq Z\otimes
L^{\gamma_2}$ implies that $\N_{N/F}$ and $\N_{Z/F}$ become isometric
after scalar extension to $L^{\gamma_2}$, hence
by~\cite[(34.9)]{EKM} the form $\N_{N/F}\perp-\N_{Z/F}$ is Witt
equivalent to a multiple of $\N_{L^{\gamma_2}/F}$. But since
$L^{\gamma_2}$ is a quadratic neat  subfield of $(B,\tau)$,
\Cref{L:discneatquadnorm} implies that $\la 1,-d\ra\otimes
\N_{L^{\gamma_2}/F}$ is hyperbolic.
Hence $\la 1,-d\ra\otimes(\N_{N/F}\perp-\N_{Z/F})$ is hyperbolic. 
Therefore we have $\la1,-d\ra\otimes\N_{N/F}\simeq\la1,-d\ra\otimes\N_{Z/F}$.

$(3)$\, Refreshing the notation, we set $(A,\s)=(B,\tau)\otimes(Q,\can_Q)$, which is an $F$-algebra with symplectic involution.
We fix a quadratic \'etale $F$-subalgebra $Z$ of $Q$.
The nontrivial $F$-automorphism of $Z$ extends to $\Int_Q(j)$ for some $j\in\mg{Q}$, and we obtain that  $$Q=Z\oplus Zj$$
and $j^2\in\mg{F}$.
Let $A'=\C_A(Z)=B\otimes Z$ and note that $A'$ is $\s$-stable. 
We set $\s'=\s|_{A'}$.
Then $(A',\s')$ is an $F$-algebra with unitary involution.
Note that $\kap(A,\s)=\kap(A',\s')=4$ and $L$ is a biquadratic neat  $F$-subalgebra of $(A',\s')$ and of $(A,\s)$.  
By $(2)$ we have
  $$\Pf_{\s'}\simeq \la1,-d\ra\otimes\N_{Z/F}\,.$$
 
We set 
\begin{eqnarray*}
W\, & = & \{x\in \Symd(\s)\mid x\ell=\gamma_1(\ell)x\text{ for all
          }\ell\in L\}\quad\mbox{ and }\\ 
W'& = & \{x\in \Symd(\s')\mid x\ell=\gamma_1(\ell)x\text{ for all
        }\ell\in L\} \,. 
\end{eqnarray*}
Since $\s'$ is unitary we have $\Symd(\s')=\Symm(\s')$ and thus $W'=W\cap A'$.
By the definition,
$\Pf_\s$ is similar to the quadratic form $q\colon W\to F,x\mapsto
s(x^2)$, and 
 $\Pf_{\s'}$ is similar to $q'\colon W'\to F, x\mapsto s(x^2)$, which
 is the restriction of $q$ to $W'$.  
 
We first look at the case where $q'$ is isotropic. In this case
$\Pf_{\s'}$ and $\Pf_\s$ are isotropic, and hence hyperbolic, because
they are Pfister forms. 
Since $\Pf_{\sigma'}\simeq\la1,-d\ra\otimes\N_{Z/F}$, and since
$\N_{Z/F}$ is a subform of $\Nrd_Q$, it follows that the $3$-fold
Pfister form $\la 1,-d\ra \otimes \Nrd_Q$ is hyperbolic, and hence
isometric to $\Pf_\s$. 

We may now assume for the rest of the proof that $q'$ is anisotropic.
Note that $\Int_{A}(j)$ commutes with $\s'$, because $\s(j)=-j$ and $j^2\in Z= \C(A')$.
Moreover $j\in\C_A(L)$. 
It follows that $W'$ is preserved under $\Int_A(j)$.

Note that $L\cap W'=0$ and $L\oplus W'\subseteq\C_{A'}(K)$.
However $\Symm(\tau)\not\subseteq \C_A(K)$, so in particular $\Symm(\tau)\neq L\oplus W'$. 
Since $\dim_F\, L\oplus W'=8=\dim_F\Symm(\tau)$ and $L\subseteq \Symm(\tau)$, we obtain that 
$W'\not\subseteq\Symm(\tau)=\Symm(\s')\cap B$. As $W'\subseteq \Symm(\s')$, we conclude that $W'\not\subseteq B$.

Since $Q=Z\oplus jZ$ and $W'\subseteq A'=\C_A(Z)$, it follows that 
there exists $w_0\in W'$ such that $jw_0\neq w_0j$. 
Let
  $w=jw_0j^{-1}-w_0\in W'$. Then $w\neq0$ and
\begin{eqnarray*}
    wj=jw_0-w_0j=jw_0+\s(jw_0) & \in & \Symd(\s)\,.
\end{eqnarray*}
 Moreover, $wj\in W$ because $w\in W'\subseteq W$, $j\in \C_A(L)$ and $jw=-wj$.

We fix an element $z\in Z\setminus F$ with $z^2-z\in F$. 
Then $jzj^{-1}=1-z$.

For every $w'\in W'$
  we have
  \begin{equation}
    \label{eq:4}
    (w'wj+wjw')z=(1-z)(w'wj+wjw').
  \end{equation}
On the other hand, since $x^2\in K$ for every $x\in W$,
  it follows that
  \[
    w'wj+wjw'=(w'+wj)^2-{w'}^2-(wj)^2\in K,
  \]
  whereby
  \begin{equation}
    \label{eq:5}
    (w'wj+wjw')z=z(w'wj+wjw')
  \end{equation}
  because $K\subseteq A'=\C_A(Z)$.  
By comparing~\eqref{eq:4} and \eqref{eq:5}, we obtain for every $w'\in W'$ that $w'wj+wjw'=0$.
This proves that $wj$ lies in the orthogonal complement of $W'$ with respect to the quadratic form $q$. 

We set $a=q(w)$. As $w\in W'\setminus\{0\}$ and $q'=q|_{W'}$ is anisotropic, we have that $a\in\mg{F}$.
Since $q|_{W'}$ is similar to $\Pf_{\s'}$, we obtain that 
$\Pf_{\s'}\simeq a q'$.
Similarly, since $q$ is similar to $\Pf_\s$ and represents $a$, we obtain that 
$\Pf_\s\simeq aq$.

Set $b=j^2$. Then $b\in\mg{F}$ and $\Nrd_Q\simeq \la 1,-b\ra\otimes \N_{Z/F}$.
We further have
$$
    q(wj)=s\bigl((wj)^2\bigr) = s(-bw^2)=-bq(w)=-ab\,.
$$
Since $wj$ lies in the orthogonal complement of $W'$ with respect to $q$, it follows that 
$q'\perp\la -ab\ra$ is a subform of $q$.
Therefore $\Pf_{\s'}\perp\la -b\ra$ is a subform of $\Pf_\s$.
On the other hand, having $\Pf_{\s'}\simeq \la 1,-d\ra\otimes \N_{Z/F}$ and 
$\la 1,-b\ra\otimes \N_{Z/F}\simeq \Nrd_Q$, we also have that $\Pf_{\s'}\perp\la -b\ra$ is a subform of $\la 1,-d\ra\otimes \Nrd_Q$.
Hence the quadratic $3$-fold Pfister forms $\Pf_\s$ and $\la 1,-d\ra\otimes \Nrd_Q$ share a common $5$-dimensional subform.
In view of \cite[Lemma~23.1]{EKM} this readily yields that they are isometric.
\end{proof}

\begin{rem}
In view of \Cref{P:cap4unitary-invarel}, one can derive part~$(2)$ of \Cref{prop:examples} alternatively
  from the description of the Brauer class of the discriminant algebra
  of $(B,\tau)\otimes(Z,\can_{Z/F})$ in \cite[(10.33)]{BOI}.
\end{rem}

We round up by computing the discriminant Pfister form in some special cases of \Cref{prop:examples} where the algebra $B$ is split.

\begin{exs}
  Let $B=\End_F(V)$ for some $4$-dimensional $F$-vector space $V$.
  Let $\beta\colon V\times V\to F$ be a nondegenerate symmetric
  bilinear form over $F$, and let $\ad_\beta$ denote the adjoint
  involution on $\End_F(V)$, which is determined by   
$$\hspace{2cm}\beta(u,f(v))=\beta(\ad_\beta(f)(u),v)\quad\mbox{ for all }f\in\End_F(V), u,v\in V.$$
  Let $d\in \mg{F}$ be the determinant of $\beta$ (determined up to a
  square factor). Applying \Cref{prop:examples}, we obtain the
  following results:
\begin{enumerate}[$(1)$]
\item If $\car(F)\neq 2$, then $\Pf_{\ad_\beta}\simeq \la 1,-d\ra$.
\item 
For
$(A,\s)=(\End_F(V),\ad_\beta)\otimes (Z,\can_{Z/F})$, where $Z$ is a quadratic \'etale $F$-algebra, we obtain that 
$\Pf_\s\simeq \la 1,-d\ra\otimes \mathsf{N}_{Z/F}$.  
\item Let $Q$ be an $F$-quaternion algebra. For
$(A,\s)=(\End_F(V),\ad_\beta)\otimes (Q,\can_Q)$, we obtain that
that $\Pf_\s\simeq\la 1,-d\ra\otimes \Nrd_{Q}$. 
\end{enumerate}
\end{exs}

\bibliographystyle{plain}

\end{document}